\documentclass{article}
\usepackage[utf8]{inputenc}
\usepackage[english]{babel}
\usepackage{color}
\usepackage{mathtools}
\usepackage[margin=1in]{geometry} 
\usepackage{amsmath,amsthm,amssymb, graphicx, multicol, array}
\usepackage[colorlinks=true, citecolor=blue, linkcolor=magenta]{hyperref}

\theoremstyle{plain}
\newtheorem{theorem}{Theorem}[section]
\newtheorem{corollary}{Corollary}[theorem]
\newtheorem{lemma}[theorem]{Lemma}

\theoremstyle{definition}
\newtheorem{definition}[theorem]{Definition}

\newtheorem{prop}[theorem]{Proposition}
\newtheorem{conjecture}[theorem]{Conjecture}
\newtheorem*{conjecture*}{Conjecture}

\theoremstyle{remark}
\newtheorem*{remark}{Remark}

\title{The Hilbert Series of the Irreducible Quotient of the Polynomial Representation of the Rational Cherednik Algebra of Type $A_{n-1}$ in Characteristic $p$ for $p|n-1$}
\author{Merrick Cai, Daniil Kalinov}
\date{March 2021}

\begin{document}

\maketitle
\begin{abstract}
	We study the irreducible quotient $\mathcal{L}_{t,c}$ of the polynomial representation of the rational Cherednik algebra $\mathcal{H}_{t,c}(S_n,\mathfrak{h})$ of type $A_{n-1}$ over an algebraically closed field of positive characteristic $p$ where $p|n-1$. In the $t=0$ case, for all $c\ne 0$ we give a complete description of the polynomials in the maximal proper graded submodule $\ker \mathcal{B}$, the kernel of the contravariant form $\mathcal{B}$, and subsequently find the Hilbert series of the irreducible quotient $\mathcal{L}_{0,c}$. In the $t=1$ case, we give a complete description of the polynomials in $\ker \mathcal{B}$ when the characteristic $p=2$ and $c$ is transcendental over $\mathbb{F}_2$, and compute the Hilbert series of the irreducible quotient $\mathcal{L}_{1,c}$. In doing so, we prove a conjecture due to Etingof and Rains completely for $p=2$, and also for any $t=0$ and $n\equiv 1\pmod{p}$. Furthermore, for $t=1$, we prove a simple criterion to determine whether a given polynomial $f$ lies in $\ker \mathcal{B}$ for all $n=kp+r$ with $r$ and $p$ fixed.
\end{abstract}
\tableofcontents
\section{Introduction}

The main object of our study in the current paper is the rational Cherednik algebra of type $A_{n-1}$, which we will denote by $\mathcal{H}_{t,c}(S_{n},\mathfrak{h})$ or simply by $\mathcal{H}_{t,c}(n)$. The Cherednik algebras, also known as Double Affine Hecke Algebras (DAHA), were introduced by Cherednik in \cite{cherednik1993macdonald} as a tool in his proof of Macdonald’s conjectures about orthogonal polynomials for root systems. Since then Cherednik algebras have appeared in many different mathematical contexts and showed their independent significance. In particular, they are directly linked with exactly solvable models in physics, especially quantum Calogero-Moser systems (see \cite{etingof2007calogero}), and quantum KZ equations (see \cite{cherednik1992kz}). In \cite{cherednik2005double}, Cherednik gives a more thorough exposition of the applications of DAHA in various mathematical areas, such as harmonic analysis, topology, elliptic curve theory, Verlinde algebras, Kac-Moody algebras, and more. Another good source on general theory of Cherednik algebras is \cite{etingof2010lecture}.

Representation theory of rational Cherednik algebras over the fields of characteristic zero was well studied, particularly in \cite{gordon2003baby} (in which the Hilbert series of irreducible representations is computed as well).

It is a topic of current research to understand the behaviour of irreducible representations of Cherednik algebras in positive characteristic (for example see \cite{balagovic2013representations}, \cite{devadas2016polynomial}, \cite{devadas2014representations}). Our work can be seen as a follow up on the article \cite{devadas2016polynomial}. In a similar fashion we restrict ourselves from the general rational Cherednik algebra $\mathcal{H}_{t,c}(\mathfrak{h},G)$, to the case where $G = S_n$, $\mathfrak{h}$ is a reflection representation of $S_n$ and $c$ is generic, but we also consider the case $t=0$. In their paper Devadas and Sun have proven the formula for the Hilbert polynomial of the irreducible quotient of the polynomial representation $\mathcal{L}_{t,c}(\text{triv})$ for $p | n$. In our paper we work on the next case $n = kp+1$. In this case we prove the formula for the Hilbert polynomial of $L_{t,c}(\text{triv})$ for any pair $(p,n)$ in the case $t=0$ and for $p=2$ in the case $t=1$ and generic $c$. We also present Conjecture~\ref{conj} due to Etingof and Rains for the Hilbert series in the general case $n = kp +r$, which holds for all of the cases that we, Devadas, and Sun have studied.
\begin{conjecture*}[Etingof, Rains]
Let $n=kp+r$, $0\le r<p$, $[k]_z=\frac{1-z^k}{1-z}$, $[k]_z!=[k]_z[k-1]_z\dotsm [1]_z$, $Q_r(n,z)=\binom{n-1}{r-1}z^{r+1}+\sum_{i=0}^{r}\binom{n-r-2+i}{i}z^i$, and $c$ be generic. The Hilbert series for $\mathcal{L}_{t,c}$ is of the form $$h_{\mathcal{L}_{0,c}}(z)=[r]_z![p]_zQ_r(n,z)\hspace{3mm}\text{and}\hspace{3mm} h_{\mathcal{L}_{1,c}}(z)=[p]_z^{n-1}[r]_{z^p}![p]_{z^p}!Q_r\left(n,z^p\right).$$
\end{conjecture*} Note that $h_{\mathcal{L}_{1,c}}(z)=[p]_z^{n-1}h_0\left(z^p\right)$, which is discussed in \cite{balagovic2013representations}.

In Section 1, we give an overview of the background, terminology, and past results in the representation theory of rational Cherednik algebras, particularly those which are relevant for the case that we work with. In Section 2, we prove Theorem~\ref{th2.34} which solves the case $t=0$ and $p|n-1$. In Section 3, we prove Theorem~\ref{th3.9} which introduces a simple criterion to determine whether a polynomial is in the maximum graded submodule of the polynomial representation (later defined as $\ker\mathcal{B}$), and then prove Theorem~\ref{th3.15} which solves the case $t=1$ over a field of characteristic $2$ and $n$ odd.

\subsection{Preliminaries}

We will adopt notation from \cite{balagovic2013representations}.

Fix an algebraically closed field $\Bbbk$ of characteristic $p$ for some prime $p$, and fix a positive integer $n>1$. Fix $t,c\in \Bbbk$. Let $S_n$ be the symmetric group on $n$ elements, and $\sigma_{ij}$ be the transposition swapping $i$ and $j$. Consider the $n$-dimensional permutation representation of $S_n$, a vector space $V$ spanned by $y_1,y_2,\dots ,y_n$ over $\Bbbk$, and its dual space $V^*$ with dual basis $x_1,x_2,\dots ,x_n$. Then consider the subrepresentation $\mathfrak{h}=\text{Span}\{y_{i}-y_{j}|i,j\in [n]\}$ over $\Bbbk$ and its dual $\mathfrak{h}^*=V^*/(x_1+x_2+\dots +x_n)$. Denote by $T(\mathfrak{h}\oplus \mathfrak{h}^*)$ the tensor algebra of $\mathfrak{h}\oplus \mathfrak{h}^*$.

\begin{definition}
\label{df1.1}
The \textit{rational Cherednik algebra of type $A_{n-1}$}, or $\mathcal{H}_{t,c}(S_n,\mathfrak{h})$, is the quotient of $\Bbbk S_n\ltimes T(\mathfrak{h}\oplus \mathfrak{h}^*)$ by the relations 
\begin{itemize}
    \item $[x_i,x_j]=0$,
    \item $[y_{i}-y_{j},y_{\ell}-y_{k}]=0$,
    \item $[y_{i}-y_{j},x_i]=t-c\sigma_{ij}-c\sum_{k\ne i}\sigma_{ik}$,
    \item $[y_{i}-y_{j},x_k]=c\sigma_{ik}-c\sigma_{jk}$ for $k\ne i,j$.
\end{itemize}
\end{definition}
\begin{remark}
One can also work with $V$ and $V^*$ instead of $\mathfrak{h}$ and $\mathfrak{h}^*$, to define $\mathcal{H}_{t,c}(S_n,V)$. For $p\nmid n$, the Hilbert series of $\mathcal{L}_{t,c}$ (defined in Definition~\ref{d1.11}) are related via
\begin{align*}
h_{\mathcal{L}_{0,c}(S_n,V)}(z)&=h_{\mathcal{L}_{0,c}(S_n,\mathfrak{h})}(z),\\
h_{\mathcal{L}_{1,c}(S_n,V)}(z)&=\left(1+z+\dots +z^{p-1}\right)h_{\mathcal{L}_{1,c}(S_n,\mathfrak{h})}(z).\\
\end{align*}
\end{remark}
Consider $S\mathfrak{h}$, the symmetric algebra of $\mathfrak{h}$, which we can think about as the subalgebra in the algebra of polynomials in $y_i$, generated by the differences $y_i-y_j$ for distinct $i,j$. Consider also $S\mathfrak{h}^{*}$ the symmetric algebra of $\mathfrak{h}^*$, which we can think about as the algebra of polynomials in $x_i$ modulo the relation $(x_1+\dots +x_n)$; i.e., $S\mathfrak{h}^{*}\cong \Bbbk[x_1,\dots ,x_n]/(x_1+\dots +x_n)$.

In \cite{devadas2016polynomial}, the PBW theorem is stated for $\mathcal{H}_{t,c}(S_n,\mathfrak{h})$.

\begin{theorem}[PBW\footnote{PBW stands for Poincare-Birkhoff-Witt, and this case (for Cherednik algebras) is a generalization of the famous theorem for Lie algebras.}]
\label{th1.1}
We have the decomposition $\mathcal{H}_{t,c}(S_n,\mathfrak{h})\simeq S\mathfrak{h}\otimes_{\Bbbk}\Bbbk[S_n]\otimes_{\Bbbk}S\mathfrak{h}^*$ as vector spaces.
\end{theorem}

We can introduce a $\mathbb{Z}$ grading on $\mathcal{H}_{t,c}$ by setting $\deg y=-1$ for $y\in\mathfrak{h}$, $\deg x=1$ for $x\in\mathfrak{h}^*$, and $\deg \sigma=0$ for $\sigma\in S_n$.

Since $\mathcal{H}_{t,c}(S_n,\mathfrak{h})\cong \mathcal{H}_{at,ac}(S_n,\mathfrak{h})$ for any $a\in \Bbbk^\times$, it suffices to study the cases $t=0$ and $t=1$.
\begin{definition}
\label{df1.2}
For parameters $t,c$, the \textit{Dunkl operator} is defined as $$D_{y_i}=t\partial_{x_i}-c\sum_{k\ne i}(x_i-x_k)^{-1}(1-\sigma_{ik})\in \text{End} (S\mathfrak{h}^*).$$
\end{definition}
\begin{remark}
Define $D_{y_{i}-y_{j}}=D_{y_i}-D_{y_j}$. This uniquely extends to a homomorphism $S\mathfrak{h}\to \text{End}(S\mathfrak{h}^*)$, since the $D_{y_i}$ commute.
\end{remark}
7

Define a structure of an $\mathcal{H}_{t,c}$-representation on $S\mathfrak{h}^*$ by sending $y_{i}-y_{j}\mapsto D_{y_{i}-y_{j}}$, $\sigma\mapsto\sigma$ (with the natural action on $S \mathfrak{h}^*$), and $x_i\mapsto x_i$ (acting by multiplication). The Dunkl operators satisfy the same commutator relations given in Definition~\ref{df1.1}, which means that this is indeed a representation (see \cite{etingof2010lecture}, Proposition 2.14 and Theorem 2.15).

\subsection{Verma Modules}
There is another way to define a polynomial representation of $\mathcal{H}_{t,c}$, and that is via Verma modules.

Consider the trivial representation $\Bbbk$ of $\Bbbk S_n\ltimes S\mathfrak{h}$; $S_n$ acts by $1$ and $y_{i}-y_{j}$ acts by $0$.

\begin{definition}
\label{df1.3}
The \textit{Verma module} is the induced $\mathcal{H}_{t,c}(S_n,\mathfrak{h})$-module $$\mathcal{M}_{t,c}(S_n,\mathfrak{h},\Bbbk)=\mathcal{H}_{t,c}(S_n,\mathfrak{h})\otimes_{\Bbbk S_n\ltimes S\mathfrak{h}}\Bbbk.$$ We will refer to it as $\mathcal{M}_{t,c}$.
\end{definition}

\begin{prop}
\label{p1.2}
The Verma module $\mathcal{M}_{t,c}$ is isomorphic to $S\mathfrak{h}^{*}$ as vector spaces.
\end{prop}
\begin{proof}
We have $$\mathcal{M}_{t,c}=\mathcal{H}_{t,c}(S_n,\mathfrak{h})\otimes_{\Bbbk S_n\ltimes S\mathfrak{h}}\Bbbk=\Bbbk S_n\ltimes (S\mathfrak{h}\otimes S\mathfrak{h}^*)\otimes_{\Bbbk S_n\ltimes S\mathfrak{h}}\Bbbk$$$$\implies f(\mathbf{x})\sigma q(\mathbf{y})\otimes 1=f(\mathbf{x})\otimes \sigma q(\mathbf{y})1=f(\mathbf{x})\otimes q(0),$$ where $f,q$ are polynomials, $\mathbf{x}$ and $\mathbf{y}$ are vectors $(x_1,\dots ,x_n)$ and $(y_1,\dots ,y_n)$, and $\sigma\in S_n$. So, by the PBW theorem for Cherednik algebras, $\mathcal{M}_{t,c}$ has a basis of elements of $S\mathfrak{h}^*$ (polynomials in $x_i$). 
\end{proof}
\begin{remark}
The Verma module $\mathcal{M}_{t,c}$ has a grading by degree, setting $\deg x_i=1$, similar to that of $\mathcal{H}_{t,c}(S_n,\mathfrak{h})$. Note that the isomorphism in Proposition~\ref{p1.2} is that of graded vector spaces.
\end{remark}
This shows that $\mathcal{M}_{t,c}\cong S\mathfrak{h}^*$ as graded vector spaces, but they are also isomorphic as representations. We have the map $y_i\rightarrow D_{y_i}$, since the action of $D_{y_i}$ and $y_i$ are given by the same relations. We also have the following identification:
\begin{prop}
\label{p1.3}
We have an isomorphism $\mathcal{H}_{t,c}(S_n,\mathfrak{h})^{opp}\cong \mathcal{H}_{t,c}(S_n,\mathfrak{h}^*)$.
\end{prop}

\begin{definition}
\label{df1.4}
The contravariant form $\mathcal{B}:\mathcal{M}_{t,c}(S_n,\mathfrak{h},\Bbbk)\times \mathcal{M}_{t,c}(S_n,\mathfrak{h}^*,\Bbbk)\rightarrow \Bbbk$ is a bilinear form satisfying the following properties:
\begin{itemize}
    \item It is $S_n$-invariant: for $\sigma \in S_n$, then $\mathcal{B}(\sigma f,\sigma q)=\mathcal{B}(f,q)$.
    \item For $x\in \mathfrak{h}^*$, $f\in \mathcal{M}_{t,c}(\mathfrak{h})$, $q\in \mathcal{M}_{t,c}(\mathfrak{h}^*)$, then $\mathcal{B}(xf,q)=\mathcal{B}(f,D_{x}(q))$.
    \item For $y\in \mathfrak{h}$, $f\in \mathcal{M}_{t,c}(\mathfrak{h})$, $q\in \mathcal{M}_{t,c}(\mathfrak{h}^*)$, then $\mathcal{B}(f,yq)=\mathcal{B}(D_{y}(f),q)$.
    \item The form is zero on elements of different degrees; i.e., if $f\in \mathcal{M}_{t,c}(\mathfrak{h})_{i}$ and $q\in \mathcal{M}_{t,c}(\mathfrak{h}^*)_j$ for $i\ne j$, then $\mathcal{B}(f,q)=0$.
    \item If $f\in \mathcal{M}_{t,c}(\mathfrak{h})_{0}$ and $q\in \mathcal{M}_{t,c}(\mathfrak{h}^*)_0$, then $\mathcal{B}(f,q)=f\cdot q$.
\end{itemize}
\end{definition}
Effectively, this contravariant form defines a bilinear form $\mathcal{B}:S\mathfrak{h}\times S\mathfrak{h}^*\rightarrow \Bbbk$ satisfying $\mathcal{B}(1,1)=1$, $\mathcal{B}(1,x_i)=0$, and $\mathcal{B}(f(y),q(x))=\mathcal{B}(1,D_{f(y)}(q(x)))=[x^0]f(D_y)q(x)$ where $[x^0]$ denotes the constant term when $f(D_y)\in S\mathfrak{h}$ acts on $q(x)\in S\mathfrak{h}^*$.
\begin{definition}
\label{df1.5}
Define an $\mathcal{H}_{t,c}(S_n,\mathfrak{h})$ representation $\mathcal{L}_{t,c}=\mathcal{M}_{t,c}/\ker \mathcal{B}$, where $\ker \mathcal{B}=\{x\in S\mathfrak{h}^*|\mathcal{B}(y,x)=0\hspace{2mm}\forall \hspace{2mm}y\in S\mathfrak{h}\}$.
\end{definition}
Note that $\ker \mathcal{B}$ is a subrepresentation and therefore also an ideal in the algebra of polynomials.
\begin{lemma}
\label{kerB}
For a fixed $f\in S\mathfrak{h}^*$ with no constant term, if $D_{y_i-y_j}f\in \ker \mathcal{B}$ for all $i,j$, then $f\in \ker\mathcal{B}$.
\end{lemma}
\begin{proof}
It suffices to prove that $\mathcal{B}(y,f)=0$ for all $y\in S\mathfrak{h}$. Since $y\in S\mathfrak{h}$, there exist polynomials $t_{ij}\in \Bbbk[y_1,\dots,y_n]$ such that $y=c+\sum_{i,j}(y_i-y_j)t_{ij}$ for $c\in \Bbbk$. By linearity of $\mathcal{B}$, we have $$\mathcal{B}(y,f)=\mathcal{B}(c,f)+\sum_{i,j}\mathcal{B}((y_i-y_j)t_{ij},f)=0+\sum_{i,j}\mathcal{B}(t_{ij},D_{y_i-y_j}f)=\sum_{i,j}0=0,$$ since by hypothesis $c$ is in the $0^{\text{th}}$ graded component and $f$ is not, and $D_{y_i-y_j}f\in \ker\mathcal{B}$ for all $i,j$.
\end{proof}
\begin{definition}
\label{df1.6}
Define the \textit{Baby Verma module} $\mathcal{N}_{t,c}(S_n,\mathfrak{h},\Bbbk)$ as follows:
\begin{itemize}
    \item If $t=1$, then $\mathcal{N}_{1,c}=\mathcal{M}_{1,c}/\left(\left(S\mathfrak{h}^*\right)^{S_n}\right)_{+}^p \mathcal{M}_{1,c}$, or $S\mathfrak{h}^*$ modulo the ideal generated by the $S_n$-invariant polynomials of positive degree raised to the $p^{\text{th}}$ power.
    \item If $t=0$, then $\mathcal{N}_{0,c}=\mathcal{M}_{0,c}/\left(\left(S\mathfrak{h}^*\right)^{S_n}\right)_+ \mathcal{M}_{0,c}$, or $S\mathfrak{h}^*$ modulo the ideal generated by the $S_n$-invariant polynomials of positive degree.
\end{itemize}
\end{definition}
It follows that $\mathcal{L}_{t,c}=\mathcal{N}_{1,c}/\ker \mathcal{B}$, because $\left(\left(S\mathfrak{h}^*\right)^{S_n}\right)_{+}^p \mathcal{M}_{1,c}\subset \ker \mathcal{B}$.

We have the following statements from, e.g., \cite{balagovic2013representations}:
\begin{enumerate}
    \item $\left((S\mathfrak{h})^{S_n}\right)_+$ is finitely generated over $\Bbbk$. (Fundamental theorem on symmetric polynomials)
    \item All $\mathcal{N}_{t,c}$ (and thus $\mathcal{L}_{t,c}$) are finite dimensional.
    \item $\ker \mathcal{B}$ is a maximal proper graded submodule of $\mathcal{M}_{t,c}$.
    \item $\mathcal{L}_{t,c}$ is irreducible.
\end{enumerate}

\begin{definition}
\label{d1.11}
We define the Hilbert series of an $\mathbb{N}$-graded module $M$ to be $h_{M}(z)=\sum_{i\ge 0}\dim M[i]z^i$, where $M[i]$ is the $i^{\text{th}}$ graded component of $M$.
\end{definition}
The quotient $\mathcal{L}_{t,c}$ inherits the grading from $\mathcal{M}_{t,c}$, hence we assign to it the Hilbert series $h_{\mathcal{L}_{t,c}}(z)=\sum_{i\ge 0}\dim \mathcal{L}_{t,c}[i]z^i$. In the general case, Etingof and Rains present the following (yet unpublished) conjecture for the Hilbert series. Let $n=kp+r$, $0\le r<p$, $$[k]_z=\frac{1-z^k}{1-z},\hspace{3mm} [k]_z!=[k]_z[k-1]_z\dotsm [1]_z,\hspace{2mm} Q_r(n,z)=\binom{n-1}{r-1}z^{r+1}+\sum_{i=0}^{r}\binom{n-r-2+i}{i}z^i.$$

\begin{conjecture}[Etingof, Rains]
\label{conj}
The Hilbert series for $\mathcal{L}_{t,c}$, with $c$ generic, is of the form $$h_{\mathcal{L}_{0,c}}(z)=[r]_z![p]_zQ_r(n,z)\hspace{3mm}\text{and}\hspace{3mm} h_{\mathcal{L}_{1,c}}(z)=[p]_z^{n-1}[r]_{z^p}![p]_{z^p}!Q_r\left(n,z^p\right).$$
\end{conjecture}
\begin{remark}
In the case $t=0$, we merely need $c\ne 0$ to be generic.
\end{remark}

\begin{definition}
\label{df1.7}
A \textit{singular polynomial} is a polynomial $f\in S\mathfrak{h}^*$ which lies in the simultaneous kernel of all Dunkl operators $D_{y_{i}-y_{j}}$, i.e. $D_{y_{i}-y_{j}}f=0$ for all $i,j$.
\end{definition}

The singular polynomials generate a submodule lying in $\ker \mathcal{B}$, thus (in positive characteristic) we would like to find such generators to understand $\ker\mathcal{B}$. This would allow us to understand $\mathcal{L}_{t,c}$. 

\subsection{Characteristic 0}
The singular polynomials for characteristic $0$ are known; see, for example, \cite{etingof2010lecture}.
\begin{prop}
\label{p1.4}
If $\text{char}\hspace{1mm}\Bbbk=0$ and $c=\frac{r}{n}$ for some $r$ not divisible by $n$, then the singular polynomials for $t=1$ are $\text{Res}_{\infty}\left[\frac{dz}{z-x_j}\prod_{i=1}^{n}(z-x_i)^c\right]$ for $j=1,2,\dots ,n-1$. (See \cite{chmutova2003some}, Proposition 3.1, for original reference, or \cite{devadas2014representations}, Proposition 1.2.)
\end{prop}

The lowest-weight irreducible representations of the rational Cherednik algebra associated to $S_n$ in characteristic $0$ are studied in \cite{gordon2003baby}, and he computes their Hilbert series.

\subsection{The case where $p|n$, by Devadas and Sun}
In \cite{devadas2016polynomial}, Devadas and Sun found the Hilbert polynomial for the representation of the Cherednik algebra $\mathcal{L}_{1,c}$ where $p|n$.

Define the polynomials $$g(z)=\prod_{j=1}^{n}(1-x_iz)\hspace{20mm}\text{and}\hspace{20mm}F(z)=\sum_{m=0}^{p-1}\binom{c}{m}(g(z)-1)^m.$$ Then for $i=1,2,\dots ,n-1$, define $f_i=\left[z^p\right]\frac{F(z)}{1-x_iz}$.

Devadas and Sun showed that the polynomials $f_i$ are singular, linearly independent and homogeneous degree $p$. They also show that if $I_c=\langle f_1,\dots ,f_{n-1}\rangle \subset \mathcal{M}_{t,c}$, then $\mathcal{M}_{t,c}/I_c$ is a complete intersection for generic $c$. In doing so, they show that for generic $c$, the Hilbert series of $\mathcal{L}_{1,c}=\mathcal{M}_{1,c}/I_c$ is $h(z)=\left(\frac{1-z^p}{1-z}\right)^{n-1}$.

\subsection{Some results from Balagovic and Chen}
In \cite{balagovic2013representations}, the following Hilbert series are described.

\begin{prop}
\label{p1.5}
The Hilbert polynomial for $\mathcal{N}_{1,c}$ is $h_{\mathcal{N}_{1,c}}(z)=\frac{\left(1-z^{2p}\right)\left(1-z^{3p}\right)\dotsm \left(1-z^{np}\right)}{(1-z)^{n-1}}$ while the Hilbert polynomial for $\mathcal{N}_{0,c}$ is $h_{\mathcal{N}_{0,c}}(z)=\frac{\left(1-z^{2}\right)\left(1-z^{3}\right)\dotsm \left(1-z^{n}\right)}{(1-z)^{n-1}}$.
\end{prop}

\begin{prop}
\label{p1.6}
The Hilbert polynomial for $\mathcal L_{1,c}$ is $h_{\mathcal{L}_{1,c}}(z)=\left(\frac{1-z^p}{1-z}\right)^{n-1}h\left(z^p\right)$ for some polynomial $h$ with nonnegative integer coefficients.
\end{prop}
\begin{remark}
This differs from \cite[Prop.~3.4]{balagovic2013representations} by a factor of $\frac{1-z^p}{1-z}$ due to the fact that we use the quotient by $x_1+\dots+x_n$.
\end{remark}

\subsection{Main Results}
We find the Hilbert series for $\mathcal{H}_{t,c}(S_n,\mathfrak{h})$ over fields $\Bbbk$ of characteristic $p|n-1$. The main theorems are Theorem~\ref{th2.34} (which generalizes Theorem~\ref{th2.12}), Theorem~\ref{th3.9}, and Theorem~\ref{th3.15}. Theorem~\ref{th2.34} states that the Hilbert series for $t=0$ is $h_{\mathcal{L}_{0,c}}(z)=\left(\frac{1-z^p}{1-z}\right)\left(1+(n-2)z+z^2\right)$. Theorem~\ref{th3.9} gives a simple, computation-based criterion for whether a given polynomial $f\in\ker\mathcal{B}$ for all $n\equiv 1\pmod{p}$: one only needs to check this condition for small (and finitely many) $n$. Theorem~\ref{th3.15} states that the Hilbert series for $t=1$ and $p=2$ is $h_{\mathcal{L}_{1,c}}(z)=\left(1+z^2\right)(1+z)^{n-1}\left(1+(n-2)z^2+z^4\right)$.

\section{The case $t=0$}
Note that in this case, the Dunkl operator is just $$D_{y_i-y_j}=-c\sum_{k\ne i}\frac{1-\sigma_{ik}}{x_i-x_k}+c\sum_{\ell\ne j}\frac{1-\sigma_{j\ell}}{x_j-x_{\ell}},$$ so the parameter $c$ does not matter (so long as it is nonzero, in which case the representation is trivial) and we may assume that $c=1$.

There is a basis of $\mathcal{M}_{t,c}$ consisting of elements of $\mathbb{F}_p[x_1,x_2,\dots ,x_n]$, and hence we assume all coefficients are from $\mathbb{F}_p$.

We proceed degree by degree and analyze each subspace $\mathcal{M}_{0,c}[i]$ starting from $i=0$ and going up. We will find some polynomials which constitute a subspace $J[i]\subset\ker\mathcal{B}$ and then find bases of $\mathcal{M}_{0,c}/J[i]$. We compute the action of the Dunkl operators to explicitly show that these are not in $\ker\mathcal{B}$, hence $J[i]=\mathcal{B}[i]$.

\subsection{Characteristic $p=2$}
We will first examine the case when the characteristic is $2$. Frequently, we will make the substitution $x_n=-(x_1+x_2+\dots +x_{n-1})=x_1+x_2+\dots +x_{n-1}$.
\begin{prop}
\label{p2.1}
For $i\ne j$, the polynomials $x_i^2+x_ix_j+x_j^2$ for $i\ne j$ are singular.
\end{prop}
\begin{proof}
It suffices to prove that the action of the Dunkl operators $D_{y_1-y_r}$ for $r=2,3,\dots ,n$ on $f=x_1^2+x_1x_2+x_2^2$ results in $0$. If $r=2$, then $$D_{y_1-y_2}f=\left[\sum_{k\ne 1}\frac{1-\sigma_{1k}}{x_1-x_k}+\sum_{k\ne 2}\frac{1-\sigma_{2k}}{x_2 - x_k}\right]\left( x_1^2 + x_1 x_2 + x_2^2\right)=\sum_{k\ne 1,2} \left[-(x_1+x_2+x_k)+(x_1 + x_2 + x_k)\right]=0.$$ If $r\ne 2$ then we see that the first sum is the same, $\sum_{k\ne 1}\frac{1-\sigma_{1k}}{x_1-x_k}f=(n-2)x_1+(n-2)x_2-(x_1+x_2)=0$. The second sum, $\sum_{k\ne r}\frac{1-\sigma_{kr}}{x_r-x_k}f$, is $0$ whenever $k\ne 1,2$. But for $k=1,2$ we obtain $-2(x_1+x_2+x_r)=0$. Hence all $D_{y_1-y_r}f=0$.
\end{proof}
\begin{prop}
\label{p2.2}
The dimension of $\mathcal{L}_{0,c}[0]$ is $1$.
\end{prop}
\begin{proof}
All elements are constants.
\end{proof}

\begin{prop}
\label{p2.3}
The dimension of $\mathcal{L}_{0,c}[1]$ is $n-1$.
\end{prop}
\begin{proof}
The basis consists of $x_1,x_2,\dots ,x_{n-1}$ after the substitution for $x_n$. Suppose a singular polynomial existed $f=\sum_{i<n}a_ix_i$. Note that $D_{y_i-y_j}x_i=1$ while $D_{y_i-y_j}x_k=0$. Then $D_{y_i-y_n}f=a_i=0$, so $f=0$.
\end{proof}

\begin{prop}
\label{p2.4}
The dimension of $\mathcal{L}_{0,c}[2]$ is $n-1$.
\end{prop}
\begin{proof}
After the substitution for $x_n$, we see that $\dim \mathcal{M}_{0,c}[2]=n-1+\binom{n-1}{2}$, with a basis given by $x_i^2$ and $x_ix_j$ for $i,j<n$. The singular polynomials $x_i^2+x_ix_j+x_j^2$ for $i,j\le n-1$ (Proposition \ref{p2.1}) are all linearly independent (each contains a unique $x_ix_j$), hence they span a space of dimension $\binom{n-1}{2}$. Subtracting the dimensions shows that $\dim L_{0,c}[2]\le n-1$. Now suppose there existed another singular polynomial $f$. Substitute for $x_n$ and then remove all $x_ix_j$ terms by adding in $x_i^2+x_ix_j+x_j^2$. This new polynomial is $g=x_1^2+\dots +x_C^2$ after a permutation of indices for some $1\le C<n$. But $$-D_{y_1}g=\sum_{k>C}\frac{1-\sigma_{1k}}{x_1-x_k}g=(n-C)x_1+x_{C+1}+x_{C+2}+\dots +x_n=Cx_1+x_2+\dots +x_C.$$ Then note that $$D_{y_n}g=\sum_{k\ne n}\frac{1-\sigma_{nk}}{x_n-x_k}g=(x_1+\dots +x_C)+Cx_n.$$ Thus $D_{y_n-y_1}g=(C+1)x_1+Cx_n=x_1+C(x_2+\dots +x_{n-1})$, which is never $0$, so $\dim \mathcal{L}_{0,c}[2]=n-1$.
\end{proof}
\begin{prop}
\label{p2.5}
The dimension of $\mathcal{L}_{0,c}[3]$ is $1$.
\end{prop}
\begin{proof}
First, consider some of the possible polynomials in $\ker \mathcal{B}$. The polynomials which come from the degree $2$ singular polynomials are of the form $x_i^3+x_i^2x_j+x_ix_j^2$ and $x_i^2x_k+x_ix_jx_k+x_j^2x_k$. We will now build up the polynomials in $\ker\mathcal{B}[3]$.

\begin{lemma}
\label{l2.6}
The polynomials $x_i^3+x_j^3\in \ker \mathcal{B}$.
\end{lemma}
\begin{proof}
We have $(x_i+x_j)\left(x_i^2+x_ix_j+x_j^2\right)=x_i^3+x_j^3\in\ker \mathcal{B}$.
\end{proof}

Next, we split into two cases.
\begin{lemma}
\label{l2.7}
For all $i$, we have $x_i^3\in\ker \mathcal{B}$.
\end{lemma}
\begin{proof}
First suppose that $n\equiv 3\pmod{4}$. Consider the sum $S=\sum_{i=1}^{n-2}\sum_{j=i+1}^{n-1}\left(x_i^3+x_i^2x_j+x_ix_j^2\right)$. Clearly $S\in\ker \mathcal{B}$ and $X=x_1^3+\dots +x_n^3=\left(\sum_{h<n}x_h\right)^3+\sum_{k<n}x_k^3\in\ker \mathcal{B}$. Note that $S=x_1^3+x_3^3+x_5^3+\dots +x_{n-2}^3+X$. Therefore $x_1^3+x_3^3+x_5^3+\dots +x_{n-2}^3\in\ker \mathcal{B}$. But this is a sum of $\frac{n-1}{2}$ cubes, which is odd. Using Lemma~\ref{l2.6} to remove $\frac{n-3}{2}$ of the cubes, we find that $x_i^3\in\ker\mathcal{B}$.

Now suppose that $n\equiv 1\pmod{4}$. Consider $S$ as before, but removing the terms $x_i^3+x_i^2x_j+x_ix_j^2$ for $i,j\in\{n-1,n-2,n-3\}$. Denote this by $T$. Then adding the terms $$\left(x_{n-3}^2x_{n-4}+ x_{n-3}^2x_{n-2} +x_{n-4}x_{n-3}x_{n-2}\right) + \left(x_{n-3}^2x_{n-4} +x_{n-3}^2x_{n-1}+x_{n-3}x_{n-4}x_{n-1}\right) $$$$+ \left(x_{n-2}^2x_{n-4} + x_{n-2}^2x_{n-3} + x_{n-4}x_{n-3}x_{n-2}\right) + \left(x_{n-2}^2x_{n-4} + x_{n-2}^2x_{n-1} + x_{n-4}x_{n-2}x_{n-1}\right) $$$$+ \left(x_{n-1}^2x_{n-4} + x_{n-1}^2x_{n-3} + x_{n-4}x_{n-3}x_{n-1}\right) + \left(x_{n-1}^2x_{n-4}+x_{n-1}^2x_{n-2}+x_{n-4}x_{n-2}x_{n-1}\right) $$$$=x_{n-3}^2x_{n-2} + x_{n-3}^2x_{n-1} + x_{n-2}^2x_{n-3} + x_{n-2}^2x_{n-1}+x_{n-1}^2x_{n-2} + x_{n-1}^2x_{n-3},$$ we obtain $x_1^3+x_3^3+x_5^3+\dots +x_{n-4}^3+X\in\ker \mathcal{B}$. Once again, this yields an odd number of cubes, hence $x_i^3\in\ker \mathcal{B}$.
\end{proof}
\begin{lemma}
\label{l2.9}
The polynomials $x_ix_jx_k\in\ker \mathcal{B}$.
\end{lemma}
\begin{proof}
Consider $$
(x_j^2x_i+ x_j^2x_k +x_ix_jx_k) + (x_j^2x_i +x_j^2x_l+x_jx_ix_l) + (x_k^2x_i + x_k^2x_j + x_ix_jx_k) +$$$$ (x_k^2x_i + x_k^2x_l + x_ix_kx_l) + (x_l^2x_i + x_l^2x_j + x_ix_jx_l) + (x_l^2x_i+x_l^2x_k+x_ix_kx_l) 
$$
$$
=x_j^2x_k + x_j^2x_l + x_k^2x_j + x_k^2x_l+x_l^2x_k + x_l^2x_j
$$
and
$$
(x_ix_jx_k + x_i^2x_j + x_k^2x_j) + (x_ix_jx_k + x_i^2x_k + x_j^2x_k) + (x_ix_jx_k + x_j^2x_i + x_k^2x_i)
$$$$
=x_ix_jx_k + x_i^2x_j + x_i^2x_k + x_j^2x_i + x_j^2x_k + x_k^2x_i + x_k^2x_j.
$$
Both are in the kernel. Subtracting yields $x_ix_jx_k$, which is also in the kernel.
\end{proof}

\begin{lemma}
\label{l2.10}
All monomials of the form $x_i^2x_j$ are equivalent modulo $\ker\mathcal{B}$.
\end{lemma}
\begin{proof}
Fix a monomial $x_a^2x_b$. Then $x_a^2x_b+x_ax_bx_k+x_k^2x_b\in\ker \mathcal{B}$, which implies that $x_a^2x_b+x_k^2x_b\in\ker \mathcal{B}$ for all $k$ (since $x_ax_bx_k\in\ker\mathcal{B}$ from Lemma~\ref{l2.9}). Thus $x_a^2x_b=x_k^2x_b$ in $\mathcal{L}_{0,c}$. Similarly, $x_b^3+x_b^2x_k+x_bx_k^2\in\ker\mathcal{B}$, so $x_b^2x_k=x_bx_k^2$ in $\mathcal{L}_{0,c}$. From these two equalities, we find that $x_a^2x_b=x_i^2x_j$ for all $a,b,i,j$.
\end{proof}
All terms are either of the form $x_i^3$ or $x_i^2x_j$ or $x_ix_jx_k$. Using Lemmas~\ref{l2.7}, \ref{l2.9}, and \ref{l2.10}, we conclude that $\dim\mathcal{L}_{0,c}[3]\le 1$. But we can easily see that the dimension is not zero. If it were, then some $x_i^2x_j$ would necessarily be in $\ker \mathcal{B}$. Without loss of generality, suppose $x_1^2x_2\in\ker \mathcal{B}$. But then $D_{y_3-y_4}x_1^2x_2=x_2x_3-x_2x_4\not\in\ker\mathcal{B}$, so this is impossible.
\end{proof}
\begin{prop}
\label{p2.11}
The dimension of $\mathcal{L}_{0,c}[j]$ is $0$ for $j>3$.
\end{prop}
\begin{proof}
We show that $\ker\mathcal{B}[4]$ contains all polynomials. It suffices to check for the appearance of all monomials of the form $x_i^4,x_i^3x_j,x_i^2x_j^2,x_i^2x_jx_k,x_ix_jx_kx_{\ell}$ in $\ker\mathcal{B}$. We can obtain any monomial with at least $3$ terms from $x_ix_jx_k\in\ker \mathcal{B}[3]$ (\ref{l2.9}). We also can obtain any monomial of the form $x_i^4$ or $x_i^3x_j$ because $x_i^3\in\ker \mathcal{B}[3]$ (\ref{l2.7},\ref{l2.8}). It remains to show that we can obtain all monomials of the form $x_i^2x_j^2$. But we know that $x_i^2x_j+x_ix_j^2\in\ker \mathcal{B}[3]$ (\ref{l2.10}). Multiplying by $x_i$ yields $x_i^3x_j+x_i^2x_j^2\in\ker \mathcal{B}[4]$, and since $x_i^3x_j\in\ker \mathcal{B}[4]$, then so is $x_i^2x_j^2$, and we find that $\dim \mathcal{L}_{0,c}[4]=0$. By the properties of the contravariant form (particularly that $\ker\mathcal{B}$ is an ideal), the Proposition follows as well.
\end{proof}
Combining Propositions~\ref{p2.2}, \ref{p2.3}, \ref{p2.4}, \ref{p2.5}, and \ref{p2.11}, we conclude with the Hilbert series.
\begin{theorem}
\label{th2.12}
The Hilbert series for $\mathcal{L}_{0,c}$ when $p=2$ is $h_{\mathcal{L}_{0,c}}(z)=(1+z)\left(1+(n-2)z+z^2\right)$.
\end{theorem}
\begin{proof}
By expanding, we may compare coefficients and verify that they match.
\end{proof}

\subsection{Characteristic $p$ is odd}
We will now study the case where the characteristic $p$ is odd. We will frequently use the substitution $x_n=-x_1-x_2-\dots -x_{n-1}$.
\begin{prop}
\label{p2.13}
The dimension of $\mathcal{L}_{0,c}[0]$ is $1$.
\end{prop}
\begin{proof}
All elements are constants.
\end{proof}
\begin{prop}
\label{p2.14}
The dimension of $\mathcal{L}_{0,c}[1]$ is $n-1$.
\end{prop}
\begin{proof}
After the substitution for $x_n$, we may assume some $f=\sum_{i<n}a_ix_i$ is singular. Note that $D_{y_n}f=\sum a_i$ and $-D_{y_1}f=(n-1)a_1-(a_2+a_3+\dots +a_{n-1})=a_1-\sum_{i<n}a_i$. Then $D_{y_n-y_i}=a_i=0$, hence $f=0$ and all $x_i$ for $i<n$ span $\mathcal{L}_{0,c}[1]$.
\end{proof}
\begin{prop}
\label{p2.15}
For distinct $i,j,k\in [n]$, the polynomials $(x_j-x_k)(x_i-x_j-x_k)$ are singular.
\end{prop}
\begin{proof}
Without loss of generality, assume $i=1$, $j=2$, and $k=3$, and denote $f=(x_2-x_3)(x_1-x_2-x_3)$. Then it suffices to check the action of the Dunkl operators $D_{y_1-y_r}$ for $r=2,3,4\dots ,n$.

Notice that the cases $r=2$ and $r=3$ are the same, because the polynomial $(x_2-x_3)(x_1-x_2-x_3)$ is invariant under the operator $-\sigma_{23}$, and since the Dunkl operator is linear, they yield the same result.

First, note that $$-D_{y_1}f=\left(x_1+x_2-x_3\right)+\left(x_2-x_1-x_3\right)+(x_2-x_3)\sum_{s\ne 1,2,3}\frac{x_1-x_s}{x_1-x_s},$$$$=x_1+x_2-x_3+x_2-x_1-x_3+(n-3)(x_2-x_3)=(n-1)(x_2-x_3)=0.$$ Now we check the action of $D_{y_2}$. We find that $$-D_{y_2}f=\left(x_3-x_1-x_2\right)+2\left(x_1-x_2-x_3\right)+\sum_{s\ne 1,2,3}(x_1-x_2-x_k),$$$$=(n-1)(x_1-x_2)-\sum_{s\in [n]}x_s=0.$$ Finally, it remains to check $D_{y_r}f$ for $r>3$. Note that this leaves $$D_{y_r}f=(x_3-x_2)+(x_r+x_2-x_1)+(x_1-x_3-x_r)=0.$$ Thus $D_{y_i}f=0$ for all $i$, which implies that $(x_j-x_k)(x_i-x_j-x_k)$ is singular.
\end{proof}
\begin{prop}
\label{p2.16}
The following is a basis for the degree $2$ singular polynomials:
\begin{itemize}
    \item $(x_1-x_i)(x_2-x_1-x_i)$ for $i=3,4,\dots , n-1$
    \item $(x_j-x_2)(x_i-x_j-x_2)$ for all (unordered) combinations of $i\ne j$ with $i,j\le n-1$.
\end{itemize}
\end{prop}
\begin{proof}
Perform the substitution for $x_n$ and consider only indices between $1$ and $n-1$ inclusive. Notice that this does not affect singular polynomials which depend on $x_n$, since those are in fact a combination of singular polynomials without an $x_n$:
$$\sum_{j,k\ne i\le n-1}(x_j-x_k)(x_i-x_j-x_k)=\sum_{j,k\ne i\le n-1}(x_i-x_j)(x_k-x_i-x_j)=0.$$
\begin{lemma}
\label{l2.17}
We can obtain all singular polynomials of the form $(x_j-x_2)(x_i-x_j-x_2)$ for $3\le i,j\le n-1$ using the aforementioned basis.
\end{lemma}
\begin{proof}
Suppose for a given (unordered) pair $(i,j)$ the (ordered) polynomial $(x_j-x_2)(x_i-x_j-x_2)$ is part of the basis. Then we obtain the alternate polynomial via $$(x_i-x_2)(x_j-x_i-x_2)=(x_1-x_j)(x_2-x_1-x_j)-(x_1-x_i)(x_2-x_1-x_j)-(x_j-x_2)(x_i-x_j-x_2).$$
\end{proof}
\begin{lemma}
\label{l2.18}
We can obtain all singular polynomials containing an $x_2$ using the aforementioned basis.
\end{lemma}
\begin{proof}
Lemma~\ref{l2.17} yields all singular polynomials containing an $x_2$ but not an $x_1$. So now assume $j=1$. We show that we can obtain the polynomials $(x_j-x_2)(x_1-x_j-x_2)$ and $(x_1-x_2)(x_j-x_1-x_2)$. But $$(x_1-x_2)(x_j-x_1-x_2)=(x_j-x_2)(x_1-x_j-x_2)+(x_1-x_j)(x_2-x_1-x_j),$$ and two of those polynomials are already in the basis, hence all three are generated by the basis.
\end{proof}
It suffices to note that $$(x_j-x_k)(x_i-x_j-x_k)=(x_j-x_2)(x_i-x_j-x_2)-(x_k-x_2)(x_i-x_k-x_2).$$ The proof now follows from Lemmas~\ref{l2.17} and \ref{l2.18}.
\end{proof}
\begin{prop}
\label{p2.19}
The dimension of $\mathcal{L}_{0,c}[2]$ is $n$.
\end{prop}
\begin{proof}
From the basis in Proposition~\ref{p2.16}, we know that $\dim \mathcal{L}_{0,c}[2]\le \dim \mathcal{M}_{0,c}[2]-|\text{basis}|=\frac{n(n-1)}{2}-\frac{n^2-3n}{2}=n$. To show equality, we show that $D_{y_1-y_2}x_i^2$ for $i=1,2,\dots ,n-1$ and $D_{y_1-y_2}x_1x_2$ are all linearly independent, showing that those $n$ polynomials generate all of $\mathcal{L}_{0,c}[2]$. (Clearly, these $n$ polynomials generate the entire subspace: first, we can obtain all polynomials of the form $x_i^2$. From the singular polynomials which contain a term $x_k(x_i-x_j)$, we only need a single monomial of the form $x_ix_j$ to generate all of the other monomials of the form $x_ax_b$. This covers every possible monomial of degree $2$.)

For $r\ne 1,2$, we have
\begin{align*}
D_{y_1-y_2}x_1^2&=-x_2,\\
D_{y_1-y_2}x_2^2&=-x_1,\\
D_{y_1-y_2}x_r^2&=x_1-x_2,\\
D_{y_1-y_2}x_1x_2&=x_2-x_1.\\
\end{align*}

Suppose such a linear combination existed as $f=a_1x_1^2+a_2x_2^2+\dots +a_{n-1}x_{n-1}^2-bx_1x_2$. Then we obtain the relations $a_1+a_2=0$, and $a_1+a_3+a_4+\dots +b=0$. But by symmetry (using other Dunkl operators $D_{y_1-y_k}$ for $k\ne 2$), we obtain that $a_1+a_k=0$. Again by symmetry, we obtain that $a_i+a_j=0$, which implies that all $a_i$ are $0$. Obviously, $x_1x_2$ is not singular, so we have the conclusion.
\end{proof}
\begin{prop}
\label{p2.20}
The dimension of $\mathcal{L}_{0,c}[3]$ is $n$ for $p>3$.
\end{prop}
\begin{proof}
We perform the substitution for $x_n$ and show that the polynomials $x_i^3$ for $i<n$ and $x_a^2x_b$ for fixed $a,b<n$, combined with multiples of the singular polynomials in degree $2$, will generate all of $\mathcal{M}_{0,c}[3]$.

We first show that they generate all of $\mathcal{M}_{0,c}[3]$. Note that we have the polynomial $x_k(x_k-x_j)(x_i-x_k-x_j)=-x_k^3+x_j^2x_k+x_ix_k^2-x_ix_jx_k$ from $\ker\mathcal{B}$. Since $x_k^3$ can be produced, we can remove it, so we have $x_j^2x_k+x_ix_k^2-x_ix_jx_k$. Now take the polynomial $x_k^2x_i+x_jx_i^2-x_ix_jx_k$, which is simply the permutation $(ijk)\in S_n$ acting on the previous polynomial. Their difference yields $x_kx_j^2-x_jx_i^2$. Again by $S_n$ action, we can also produce the polynomial $x_lx_j^2-x_jx_i^2$. Their difference yields $$\left(x_kx_j^2-x_jx_i^2\right)-\left(x_lx_j^2-x_jx_i^2\right)=x_i^2x_j-x_l^2x_j.$$ Thus, if any one term of the form $x_i^2x_j$ were not in $\ker\mathcal{B}[3]$, using $x_1^2x_2$ we can produce anything of the form $a^2b$. Since all terms of the form $x_i^3$ are already produced, we only need terms of the form $x_ix_jx_k$, which can be easily obtained from $x_k(x_k-x_j)(x_i-x_j-x_k)$. Hence the dimension is at most $n$.

To show that the dimension is exactly $n$, we will show that no linear combination of $x_1^3$, $x_2^3$, \dots, $x_{n-1}^3$, $x_1^2x_2$ is in $\ker \mathcal{B}$. We will perform computations in $\mathcal{L}_{0,c}[2]$ (which was already found in Proposition~\ref{p2.19}) for the $D_{y_1-y_2}x_r^3$, subtracting the $(x_1-x_2)(x_r-x_1-x_2)$ polynomial, since that polynomial is in $\ker \mathcal{B}$ and thus is $0$ in $\mathcal{L}_{0,c}$. For $r\ne 1,2$, we have
\begin{align*}
D_{y_1-y_2}x_1^3&=x_1^2-x_1x_2-x_2^2,\\
D_{y_1-y_2}x_2^3&=x_1^2+x_1x_2-x_2^2,\\
D_{y_1-y_2}x_r^3&=x_1^2+x_1x_r-x_2x_r+x_2^2=2x_1^2,\\
D_{y_1-y_2}x_1^2x_2&=x_2^2-x_1^2.\\
\end{align*}
Suppose a linear combination $f=a_1x_1^3+\dots +a_{n-1}x_{n-1}^3+bx_1^2x_2$ is a singular polynomial; we will show that all coefficients are $0$. We must have that $D_{y_1-y_2}f=0$ in $\mathcal{L}_{0,c}[2]$, in accordance with the action of the Dunkl operators above. This means that $D_{y_1-y_2}f+\sum_{i,j,k}d_{ijk}(x_j-x_k)(x_i-x_j-x_k)=0$ in $\mathcal{M}_{0,c}[2]$. In particular, any singular polynomial from the summation introduces two $x_ix_j$ monomials, whereas there is only one from the action of $D_{y_1-y_2}f$ (which is $x_1x_2$). Therefore, by parity, all $d_{ijk}=0$ (the monomials of the form $x_ix_j$ will never cancel with each other or the $x_1x_2$ term). Thus we may only concern ourselves with the results from the above computations. From the action of $D_{y_1-y_2}$, we have $a_1=a_2$. By symmetry $a_1=a_2=\dots =a_{n-1}$, so let $a_i=a$. Then $aD_{y_1-y_2}\left(x_1^3+\dots +x_{n-1}^3\right)=2a(n-2)x_1^2-2ax_2^2=-2a\left(x_1^2+x_2^2\right)$. Comparing the coefficient of $x_1^2$ and $x_2^2$, we find that $a=-a\implies a=0$ and $b=0$ as well. We conclude that $a_1=a_2=\dots=a_{n-1}=b=0$, so no nontrivial linear combination is a singular polynomial.
\end{proof}
\begin{prop}
\label{p2.21}
The polynomials $x_i^2x_j-x_ix_j^2\in \ker \mathcal{B}$.
\end{prop}
\begin{proof}
Let $f=x_3^2x_4-x_3x_4^2$. Note that $-D_{y_2}f=-x_3^2+x_4^2+x_2x_3-x_2x_4\in\ker\mathcal{B}$. Similarly, $-D_{y_j}f=-x_3^2+x_4^2+x_jx_3-x_jx_4\in\mathcal{B}$ for all $j\ne 3,4$. Finally, it remains to compute $-D_{y_3}f=0$. (Note that $-D_{y_4}$ acts in the same way as $-D_{y_3}$ by virtue of swapping indices.) Hence $D_{y_i-y_j}f\in\ker\mathcal{B}\implies f\in\ker\mathcal{B}$ by Lemma~\ref{kerB}.
\end{proof}

\begin{prop}
\label{p2.22}
When $p=3$, the polynomials $x_i^3-x_i^2x_j+x_j^3\in\ker \mathcal{B}$.
\end{prop}
\begin{proof}
It suffices to prove that all Dunkl operators $D_{y_i-y_j}$ send $x_1^3-x_1^2x_2+x_2^3$ to a degree $2$ singular polynomial. Let $f=x_1^3-x_1^2x_2+x_2^3$. Note that $-D_{y_1}f=-x_1x_2-x_2^2+x_1x_2+x_2^2\in\ker\mathcal{B}$, and the action of $-D_{y_2}$ is exactly the same by symmetry (and by Proposition~\ref{p2.21} we may replace $x_1^2x_2$ by $x_1x_2^2$). Finally, for $j\ne 1,2$, $-D_{y_j}f=(x_2-x_j)(x_1-x_2-x_j)\in\ker\mathcal{B}$. Hence $D_{y_i-y_j}f\in\ker\mathcal{B}\implies f\in\ker\mathcal{B}$ by Lemma~\ref{kerB}.
\end{proof}
\begin{remark}
This shows that for $p=3$, the dimension of $\mathcal{L}_{0,c}[3]$ is actually $n-1$, because we do not need the polynomial $x_i^2x_j$ to be in $\mathcal{L}$ anymore. So long as all $x_i^3$ are for $i<n$, we can recover the $x_i^2x_j$.
\end{remark}

From now on, we work in $\mathcal{M}'_{0,c}=\mathcal{M}_{0,c}/\left(x_i^2x_j-x_ix_j^2\right)$ (and also with $\ker\mathcal{B}/\left(x_i^2x_j-x_ix_j^2\right)$).  Thus, we can shift exponents around in any monomial so long as all of the exponents remain positive. We will therefore not concern ourselves with specific exponents, but only with the variables that appear in the monomial.
\begin{definition}
\label{df2.1}
Denote $x_{s_1}^{e_1}x_{s_2}^{e_2}\dotsm x_{s_b}^{e_b}$ by the tuple $(s_1,\dots ,s_b)$. (The degree will be specified each time.)
\end{definition}

Denote the singular polynomials from Proposition~\ref{p2.15} as $(i)-(j)+(j,k)-(i,k)$. Notice that unless a new singular polynomial appears, then the $\ker \mathcal{B}[3]$ polynomials are either the symmetric polynomial $(1)+(2)+\dots +(n)$ or multiples of the singular polynomials, namely $(i)-(i,j)+(i,j,k)-(i,k)$ or $(i,l)-(j,l)+(j,k,l)-(i,k,l)$.
\begin{definition}
\label{df2.2}
Denote $(*)_j$ as the set of polynomials $(i)-(j)+(j,k)-(i,k)\in\mathcal{M}'_{0,c}[j]$. Similarly, define $(\dagger)_j$ as the set of polynomials $(i)-(i,j)+(i,j,k)-(i,k)\in\mathcal{M}'_{0,c}[j]$.
\end{definition}
\begin{definition}
Define $\mathcal{I}$ to be the ideal generated by the polynomials from~\ref{p2.15}.
\end{definition}
\begin{remark}
It's worth pointing out that $\mathcal{I}\subset \mathcal{B}$ and that if no new polynomials (which are sent into $\ker\mathcal{B}$ upon action by any Dunkl operator) appear in some gradation $\mathcal{M}_{0,c}[j]$, then $\mathcal{I}[j]=\ker \mathcal{B}[j]$.
\end{remark}
\begin{prop}
\label{p2.23}
In $\ker \mathcal{B}[3]$, the following polynomials constitute a basis for $(*)_3$ and $(\dagger)_3$:
\begin{itemize}
    \item[$(*)_3$]: $(i)-(j)+(j,k)-(i,k)$; choose $i,j,k$ in the same as we did for degree $2$ singular polynomials
    \item[$(\dagger)_3$]: $(i)-(i,j)+(i,j,k)-(i,k)$; choose $i<j<k$.
\end{itemize}
\end{prop}
\begin{proof}
The basis for $(*)_3$ generates all polynomials in $(*)_3$, which is proven in the same fashion as in Proposition~\ref{p2.16}. We can see that all polynomials in the basis of $(\dagger)_3$ are linearly independent with each other and the basis of $(*)_3$ because they contain a unique term $(i,j,k)$. Furthermore, we can generate all polynomials of the form $(a)-(a,b)+(a,b,c)-(a,c)$ for distinct $a,b,c\in [n]$. If $a<b<c$, then take $a=i,b=j,c=k$. If $b<a<c$, then take the polynomial $(b)-(b,a)+(b,a,c)-(b,c)\in (\dagger)$. We know that from the basis of $(*)$ we can form the polynomial $(a)-(b)+(b,c)-(a,c)$. Adding the two yields the result. Similarly, we can take any polynomial in $(\dagger)_3$ and obtain all polynomials which are permutations of its indices by adding or subtracting polynomials of the form $(i)-(j)+(j,k)-(i,k)$.
\end{proof}
\begin{prop}
\label{p2.24}
We can generate all polynomials in $\mathcal{I}[3]$ using $(*)_3$ and $(\dagger)_3$.
\end{prop}
\begin{proof}
As noted in Proposition~\ref{p2.23}, we simply need to generate all polynomials of the form $(i)-(i,j)+(i,j,k)-(i,k)$ or $(i,l)-(j,l)+(j,k,l)-(i,k,l)$, and the symmetric polynomial. The first kind, $(i)-(i,j)+(i,j,k)-(i,k)$, is exactly produced by the basis of $(\dagger)_3$. The second kind, $(i,l)-(j,l)+(j,k,l)-(i,k,l)$, can be written as $\left[(l)-(j,l)+(j,k,l)-(k,l)\right]-\left[(i)-(i,l)+(i,k,l)-(i,k)\right]+\left[(i)-(l)-(i,k)+(k,l)\right]$, and hence is also generated by $(*)_3$ and $(\dagger)_3$.

Finally, it remains to show that $(1)+(2)+\dots +(n)$ can be generated by $(*)$ and $(\dagger)$. But notice that from $(*)$, $$\sum_{k=1}^{n-2}\left[(k)-(n)+(n-1,n)-(1,n-1)\right]=(1)+(2)+\dots +(n-2)+(n)-(n-1,n)+(n-1)+(n,n-1)$$$$=(1)+(2)+\dots +(n).$$
\end{proof}

By Proposition~\ref{p2.16}, the basis for $\mathcal{I}[2]$ is given by $(*)_2$. The basis for $\mathcal{I}[3]$ is given by $(*)_3$ and $(\dagger)_3$. Now heading into higher degrees, the $(*)_j$ will always generate $(*)_{j+1}$ and $(\dagger)_{j+1}$, but the set of polynomials $(\dagger)_3$ will produce both $(\dagger)_4$ and the set of polynomials of the form $(i,l)-(i,j,l)+(i,j,k,l)-(i,k,l)$.
\begin{definition}
\label{l2.3}
Denote the set of polynomials in $\ker \mathcal{B}[j]$ of the form $(i,l_1,l_2,\dots, l_{q-3})-(i,j,l_1,\dots ,l_{q-3})+(i,j,k,l_1,l_2,\dots, l_{q-3})-(i,k,l_1,\dots ,l_{q-3})$ as $(\dagger^q)_j$ (for $q>3$).
\end{definition}
\begin{remark}
In $(\dagger^q)_j$, the $j$ is the degree, and the $q$ denotes the maximum number of distinct $x_i$ which may appear in a single element. Necessarily $q\ge 3$, since we have $i,j,k$ appearing; setting $q=3$ recovers the set $(\dagger)$. We will now focus on $q>3$.
\end{remark}
\begin{prop}
\label{p2.25}
In $\ker \mathcal{B}[j]$ for $j>3$, the following polynomials constitute a basis for $(*)_j$, $(\dagger)_j$, and $(\dagger^q)_j$ for $q=4,5,\dots ,j$:
\begin{itemize}
    \item[$(*)_j$]: $(i)-(j)+(j,k)-(i,k)$; choose $i,j,k$ in the same as in Proposition~\ref{p2.16},
    \item[$(\dagger)_j$]: $(i)-(i,j)+(i,j,k)-(i,k)$; choose $i<j<k$,
    \item[$(\dagger^q)_j$]: $(i,l_1,l_2,\dots, l_{q-3})-(i,j,l_1,\dots, ,l_{q-3})+(i,j,k,l_1,l_2,\dots, l_{q-3})-(i,k,l_1,\dots ,l_{q-3})$ for $q=4,5,\dots ,j$; choose $i<j<k<l_1<l_2<\dots <l_{q-3}$ for each $q$.
\end{itemize}
\end{prop}
\begin{proof}
We already know that the bases for $(*)_j$ and $(\dagger)_j$ are linearly independent and generate all of $(*)_j$ and $(\dagger)_j$. But we can easily see that for the bases of the $(\dagger^q)_j$'s, they are all linearly independent due to each one containing a unique term of $(i,j,k,l_1,l_2,\dots, l_{q-3})$. Hence we can inductively show that each basis for $(\dagger^q)_j$ is linearly independent with all the basis polynomials for $(*)_j$, $(\dagger)_j$, and $(\dagger^r)_j$ for $r<q$.

It thus remains to show that the basis for $(\dagger^q)_j$ can generate all of $(\dagger^q)_j$. Choose an arbitrary $q$. Then we have polynomials of the form $(i,l_1,l_2,\dots, l_{q-3})-(i,j,l_1,\dots ,l_{q-3})+(i,j,k,l_1,l_2,\dots, l_{q-3})-(i,k,l_1,\dots ,l_{q-3})$ for $q=4,5,\dots ,j$ with $i<j<k<l_1<l_2<\dots <l_{q-3}$. But note that we can arbitrarily shuffle the order of $i,j,k,l_1,\dots ,l_{q-3}$ by adding and subtracting polynomials in the basis of $(\dagger^{q-1})_j$. Thus, by induction, we have all of $(\dagger^q)_j$ for each $q$.
\end{proof}
\begin{prop}
\label{p2.26}
The polynomials $(*)_j$, $(\dagger)_j$, and $(\dagger^q)_j$ for $4\le q\le j$ generate all of $\mathcal{I}[j]$.
\end{prop}
\begin{proof}
Fix a $j>3$ (the case $j=3$ was already done in Proposition~\ref{p2.23}). Then to obtain $j+1$, each of the basis polynomials for $\mathcal{I}[j]$ are multiplied by an $x_i$, and there is a new symmetric polynomial. However this symmetric polynomial is explicitly given as follows: $$\sum_{k=1}^{n-2}\left[(k)-(n)+(n-1,n)-(1,n-1)\right]=(1)+(2)+\dots +(n-2)+(n)-(n-1,n)+(n-1)+(n,n-1),$$$$=(1)+(2)+\dots +(n).$$
Multiplying $(*)_j$ yields $(*)_{j+1}$ or $(\dagger)_{j+1}$. Multiplying $(\dagger)_j$ yields $(\dagger)_{j+1}$ or $(\dagger^4)_{j+1}$. Multiplying $(\dagger^q)_j$ yields either $(\dagger^q)_{j+1}$ or $(\dagger^{q+1})_{j+1}$, hence the result.
\end{proof}
\begin{prop}
\label{p2.27}
The polynomial $(i)-(i,j)+(j)$ is singular in degree $p$.
\end{prop}
\begin{proof}
We will use the polynomial $f=x_1^p-x_1x_2^{p-1}+x_2^p$ and first check the action of the Dunkl operator $-D_{y_1}$. We will freely replace monomials with the notation $(i,j)$.

We have that
\begin{align*}
-D_{y_1}f&=\sum_{k\ne 1}\frac{1-\sigma_{1k}}{x_1-x_k}\left(x_1^p-x_1x_2^{p-1}+x_2^p\right),\\
&=x_1x_2\frac{x_1^{p-2}-x_2^{p-2}}{x_1-x_2}+\sum_{k>2}\frac{x_1^p-x_k^p}{x_1-x_k}-x_2^{p-1}\sum_{k>2}\frac{x_1-x_k}{x_1-x_k},\\
&=x_1x_2\frac{x_1^{p-2}-x_2^{p-2}}{x_1-x_2}+x_2^{p-1}-x_1^{p-1}-x_1^{p-1}-x_2^{p-1}+\sum_{k>2}\left(x_1^{p-2}x_k+\dots +x_1x_k^{p-2}\right),\\
&=(p-2)\cdot (1,2)+(2)-(1)-(1)-(2)+(p-2)\sum_{k>2}(1,k),\\
&=-2\cdot (1,2)-2\cdot (1)-2\cdot \left[-(1)-(1,2)\right],\\
&=0.\\
\end{align*}

Since $x_1^p-x_1x_2^{p-1}+x_2^p$ is invariant under the action of $\sigma_{12}$ in $\mathcal{M}'_{0,c}$, we have $-D_{y_2}\left(x_1^p-x_1x_2^{p-1}+x_2^p\right)=0$.

Now take $j\ne 1,2$. Then we have
\begin{align*}
D_{y_j}f&=\sum_{k\ne j}\frac{1-\sigma_{jk}}{x_k-x_j}\left(x_1^p-x_1x_2^{p-1}+x_2^p\right),\\
&=\frac{1-\sigma_{1j}}{x_1-x_j}\left(x_1^p-x_1x_2^{p-1}+x_2^p\right)+\frac{1-\sigma_{2j}}{x_2-x_j}\left(x_1^p-x_1x_2^{p-1}+x_2^p\right),\\
&=(1)+(j)+(p-2)\cdot (1,3)-(2)+(2)+(j)+(p-2)\cdot (2,j)-(1,2)-(1,j)-(p-3)\cdot (1,2,j),\\
&=(1)+2\cdot (j) -3\cdot (1,j)-2\cdot (2,j)-(1,2)+3\cdot (1,2,j),\\
&=\left[(1)-(j)+(2,j)-(1,2)\right]-3\left[(j)-(1,j)+(1,2,j)-(2,j)\right]\in\ker\mathcal{B}[p-1].\\
\end{align*}
Hence all Dunkl operators $D_{y_k-y_l}$ send all degree $p$ polynomials of the form $(i)-(i,j)+(j)$ into the kernel, so they are in $\ker\mathcal{B}[p]$.
\end{proof}
\begin{prop}
\label{p2.28}
For any $j$, the dimension of $\mathcal{L}_{0,c}[j]$ is at most $n$.
\end{prop}
\begin{proof}
We will show that the $n$ polynomials $(1)$, $(2)$, \dots, $(n-1)$, and $(1,2)$, combined with $(*)_j$, $(\dagger)_j$, and $(\dagger^q)_j$ (for $4\le q\le j$) linearly generate the entire subspace of homogeneous degree $j$ polynomials.

First, we easily obtain all polynomials of the form $(i)$, since the only missing one is $(n)$ but the symmetric polynomial fills that in.

Next, we obtain all polynomials of the form $(i,j)$ from $(*)_j$, since each is of the form $(i)-(j)+(j,k)-(i,k)$. We can remove $(i)-(j)$ and we are left with $(j,k)-(i,k)$. Setting $k=1$ and $i=2$ allows us to obtain all of the form $(1,j)$ and next setting $i=1$ and $j,k$ to be anything allows us to obtain all $(j,k)$.

We can then obtain all $(i,j,k)$ from $(\dagger)_j$, since they are of the form $(i)-(i,j)+(i,j,k)-(i,k)$. Removing the necessary terms leaves us with $(i,j,k)$.

For $q>3$, any polynomial in $(\dagger^q)_j$ contains one term with $q$ distinct variables and other terms with less than $q$ distinct variables. Inductively we can remove all other terms to obtain all terms of the form $(i_1,i_2,\dots ,i_q)$. When we reach $q=j$, we are done.
\end{proof}
\begin{prop}
\label{p2.29}
For $2\le j\le p-1$, the dimension of $\mathcal{L}_{0,c}[j]$ is $n$.
\end{prop}
\begin{proof}
We showed that $\dim\mathcal{L}_{0,c}[2]=n$ in Proposition~\ref{p2.19}, so assume that $j>2$. It suffices to show that no linear combination of $(1)$, $(2)$, \dots, $(n-1)$, and $(3,4)$ is in $\ker \mathcal{B}$.

Let us examine the operator $D_{y_1-y_2}$. We perform computations strictly in $\mathcal{L}_{0,c}[j-1]$ (which has been found by the prior inductive step), adding and subtracting polynomials from $\ker \mathcal{B}$ freely. We have that
\begin{align*}
D_{y_1-y_2}x_1^j&=-\sum_{k\ne 1}\frac{1-\sigma_{1k}}{x_1-x_k}x_1^j+\sum_{k\ne 2}\frac{1-\sigma_{2k}}{x_2-x_k}x_1^j,\\
&= (j-2)\cdot (1)-(2)+(j-2)\cdot (1,2).\\
\end{align*}
We also have that
\begin{align*}
D_{y_1-y_2}x_2^j&=-\sum_{k\ne 1}\frac{1-\sigma_{1k}}{x_1-x_k}x_2^j+\sum_{k\ne 2}\frac{1-\sigma_{2k}}{x_2-x_k}x_2^j,\\
&= -(j-2)\cdot (2)+(1)-(j-2)\cdot (1,2).\\
\end{align*}
For $r\ne 1,2$, then
\begin{align*}
D_{y_1-y_2}x_r^j&=-\sum_{k\ne 1}\frac{1-\sigma_{1k}}{x_1-x_k}x_r^j+\sum_{k\ne 2}\frac{1-\sigma_{2k}}{x_2-x_k}x_r^j,\\
&=(1)-(2)+(j-2)\cdot (1,r)-(j-2)\cdot (2,r)+(j-2)\left[(1)-(2)+(2,t)-(1,r)\right],\\
&=(j-1)\cdot \left[(1)-(2)\right].\\
\end{align*}
Finally,
\begin{align*}
D_{y_1-y_2}x_3x_4^{j-1}=&-\sum_{k\ne 1}\frac{1-\sigma_{1k}}{x_1-x_k}x_3x_4^{j-1}+\sum_{k\ne 2}\frac{1-\sigma_{2k}}{x_2-x_k}x_3x_4^{j-1},\\
=&(1,3)-(2,3)+(j-3)\cdot (1,3,4)-(j-3)\cdot (2,3,4)-(j-3)\left[(1)-(1,3)+(1,3,4)-(1,4)\right],\\
=&(j-2)\cdot (1,3)-(2,3)-(j-3)\cdot (1)+(j-3)\cdot (1,4)-(j-3)\cdot (2,3,4)\\
&+(j-3)\left[(2)-(2,3)+(2,3,4)-(2,4)\right],\\
=&(j-2)\cdot (1,3)-(j-2)\cdot (2,3)-(j-3)\left[ (1)-(2)+(2,4)-(1,4)\right]\\
&+(j-3)\left[(1)-(2)+(2,4)-(1,4)\right],\\
=&(j-2)\left[(1,3)-(2,3)+(1)-(2)+(2,3)-(1,3)\right],\\
=&(j-2)\left[(1)-(2)\right].\\
\end{align*}

Suppose that we have such a polynomial, $f=a_1x_1^j+a_2x_2^j+\dots +a_{n-1}x_{n-1}^j+bx_3x_4^{j-1}$. Then we note that $a_1=a_2$ to remove the $(1,2)$ terms. By symmetry (using Dunkl operators), we have $a_1=a_2=a_3=\dots =a_{n-1}$. Obviously $a_i\ne 0$, since $x_3x_4^{j-1}$ is not singular. Now we assume without loss of generality that $a_i=\frac{1}{j-1}$, to obtain that $$D_{y_1-y_2}\left(\sum a_ix_i^j\right)=(n-2)\left[(1)-(2)\right]=(2)-(1).$$ We thus conclude that $b=-\frac{1}{j-2}$. But a quick check using the operator $D_{y_1-y_3}$ shows that $f$ is not singular after all, and hence no linear combination exists.
\end{proof}
\begin{prop}
\label{p2.30}
The dimension of $\mathcal{L}_{0,c}[p]$ is $n-1$.
\end{prop}
\begin{proof}
From Proposition~\ref{p2.28}, $(1)$, $(2)$, \dots, $(n-1)$, and $(1,2)$ generate all of $\mathcal{L}_{0,c}[p]$. But by Proposition~\ref{p2.27}, $(1)-(1,2)+(2)\in\ker\mathcal{B}$. Thus $(1,2)$ is not needed and $\dim\mathcal{L}_{0,c}[p]\le n-1$.

To show that the dimension is exactly $n-1$, we consider a linear combination $f=a_1x_1^p+a_2x_2^p+\dots +a_{n-1}x_{n-1}^p$. Using the Dunkl operator $D_{y_1-y_2}$, we compute that (for $r\ne i,j$)
\begin{align*}
D_{y_i-y_j}x_i^p&=-(j),\\
D_{y_i-y_j}x_2^p&=(i),\\
D_{y_i-y_j}x_r^p&=(i)-(j).\\
\end{align*}

This implies that $a_i=a_j$ for all $i,j$. But then $$D_{y_1-y_2}\left(x_1^p+x_2^p+\dots +x_{n-1}^p\right)=(n-2) \left[(1)-(2)\right]\ne 0.$$ Thus there does not exist such a linear combination and the dimension is exactly $n-1$.
\end{proof}
\begin{prop}
\label{p2.31}
The dimension of $\mathcal{L}_{0,c}[p+1]$ is $1$.
\end{prop}
\begin{proof}
From Proposition~\ref{p2.28}, the set $\{x_1^{p+1},x_2^{p+1},\dots, x_{n-1}^{p+1},x_1x_2^p\}$ generates $\mathcal{L}_{0,c}[p+1]$. However, note that $(i)-(i,j)+(j)\in\ker\mathcal{B}[p]$. Multiplying by $(i)$ yields $(i)-(i,j)+(i,j)=(i)\in\ker\mathcal{B}[p+1]$. Hence $\dim\mathcal{L}_{0,c}[p+1]\le 1$. To prove equality, it suffices to check that in $\mathcal{L}_{0,c}$, $D_{y_1-y_2}x_1x_2^p=(1,2)+(2)=(1,2)\ne 0$.
\end{proof}
\begin{prop}
\label{p2.32}
The dimension of $\mathcal{L}_{0,c}[p+2]$ is $0$.
\end{prop}
\begin{proof}
From Proposition~\ref{p2.31}, $x_i^{p+2}=(i)\in\ker \mathcal{B}[p+2]$, and $x_ix_j^{p+1}=(i,j)\in\ker\mathcal{B}[p+2]$. This covers all polynomials described in Proposition~\ref{p2.28}, implying that $\ker\mathcal{B}[p+2]$ contains all of $\mathcal{M}_{0,c}[p+2]$.
\end{proof}
\begin{prop}
\label{p2.33}
For all $m>p+2$, the dimension of $\mathcal{L}_{0,c}[v]$ is $0$.
\end{prop}
\begin{proof}
This follows from the fact that $\ker\mathcal{B}$ is an ideal.
\end{proof}
\begin{theorem}
\label{th2.34}
The Hilbert series for $\mathcal{L}_{0,c}$ over a field with prime characteristic $p$ is $$h_{\mathcal{L}_{0,c}}(z)=\left(\frac{1-z^p}{1-z}\right)\left(1+(n-2)z+z^2\right).$$
\end{theorem}
\begin{proof}
We simply expand and match coefficients from the previous Propositions. Note that this case also covers $p=2$ from Theorem~\ref{th2.12}.
\end{proof}

\section{The case $t=1$}

In this case, the Dunkl operator is $$D_{y_i-y_j}=\partial_{x_i}-\partial_{x_j}-c\sum_{k\ne i}\frac{1-\sigma_{ik}}{x_i-x_k}+c\sum_{\ell\ne j}\frac{1-\sigma_{j\ell}}{x_j-x_{\ell}}.$$ We will again study each gradation (by degree) explicitly from $\deg =0$ upwards and find a collection of polynomials which belong in $\ker\mathcal{B}$. We will prove that any polynomial not in the span of that collection is not in $\ker\mathcal{B}$ using explicit Dunkl operator actions. We will also utilize a result from \cite{balagovic2013representations} which constricts the form of the Hilbert series of $\mathcal{L}_{1,c}$.

\subsection{Characteristic $p=2$}
From Proposition~\ref{p1.6} (\cite{balagovic2013representations}), we know that the Hilbert series is of the form $h_{\mathcal{L}_{1,c}}(z)=(1+z)^{n-1}Q\left(z^2\right)$ for some integer polynomial $Q$. Assume $c$ is transcendental over $\mathbb{F}_2$. We will let $Q\left(z^2\right)=Q_0+Q_2z^2+Q_4z^4+\dots $ and compute term by term: each time we compute $\dim\mathcal{L}_{1,c}[d]$ for some even $d$, we can expand $h_{\mathcal{L}_{1,c}}(z)\left[z^d\right]$ to find $Q_d$.

There is a basis of $\mathcal{M}_{t,c}$ consisting of elements of $\mathbb{F}_2[c,x_1,x_2,\dots ,x_n]$, so we assume that all coefficients are from $\mathbb{F}_2[c]$.

Any polynomial $f$ can be graded by powers of $c$. We divide out by a power of $c$ so that $f\not\equiv 0\pmod{c}$. Since the Dunkl operator acts with two separate gradations, we may consider them separately.
\begin{definition}
For a given polynomial $f\in\mathbb{F}_2[c,x_1,\dots ,x_{n-1}]$, let $f=\sum_{k\ge 0}c^kf^{(k)}$ where each $f^{(k)}\in\mathbb{F}_2[x_1,\dots ,x_{n-1}]$ and $f^{(0)}\ne 0$.
\end{definition}
\begin{definition}
Denote $\alpha_{ij}=\partial_{x_i}-\partial_{x_j}$ and $\beta_{ij}=-\sum_{k\ne i}\frac{1-\sigma_{ik}}{x_i-x_k}+\sum_{\ell\ne j}\frac{1-\sigma_{j\ell}}{x_j-x_{\ell}}$, so that $D_{y_i-y_j}=\alpha_{ij}+c\beta_{ij}$.
\end{definition} We will use this notation throughout the rest of the paper. We will again use the substitution $x_n=x_1+x_2+\dots +x_{n-1}$.

\begin{prop}
\label{p3.1}
The dimension of $\mathcal{L}_{1,c}[0]$ is $1$.
\end{prop}
\begin{proof}
All elements are constants.
\end{proof}
\begin{corollary}
\label{c3.1.1}
We have that $Q_0=1$.
\end{corollary}
\begin{prop}
\label{p3.2}
The dimension of $\mathcal{L}_{1,c}[1]$ is $n-1$.
\end{prop}
\begin{proof}
Under the substitution for $x_n$, we have $\dim\mathcal{L}_{1,c}[1]\le n-1$. To show equality, assume that $f=a_1x_1+a_2x_2+\dots +a_{n-1}x_{n-1}\in\ker\mathcal{B}$ where $a_i\in\mathbb{F}_2[c]$. Then $\alpha_{in}f^{(0)}=(\partial_{x_i}-\partial_{x_n})f^{(0)}=a_i=0$ ensures that $f^{(0)}=0$, which contradicts our assumption of a nonzero constant term.
\end{proof}
\begin{remark}
We actually get this for free by the form of the Hilbert series given in \cite{balagovic2013representations}.
\end{remark}
\begin{prop}
\label{p3.3}
The dimension of $\mathcal{L}_{1,c}[2]$ is $\binom{n}{2}$.
\end{prop}
\begin{proof}
Under the substitution for $x_n$, we can concern ourselves only with polynomials in $\{x_1,x_2,\dots ,x_{n-1}\}$. We show that no singular polynomials exist, which implies that $\dim \mathcal{L}_{1,c}[2]=\binom{n}{2}$. In these $n-1$ variables, $\ker\mathcal{B}[1]=\{0\}$. Consider some singular polynomial $f$. Then $\alpha_{in}f^{(0)}=(\partial_{x_i}-\partial_{x_n})f^{(0)}=\partial_{x_i}f^{(0)}=0$ implies that it cannot contain any term of the form $x_ix_j$. Let the remaining terms be $f^{(0)}=x_1^2+x_2^2+\dots +x_C^2$ for some $C<n$. Obviously $\alpha_{ij}f^{(0)}=0$, but $\beta_{1n}f^{(0)}=Cx_n+(1-C)x_1$ which is $x_1$ or $x_1+x_2+\dots +x_{n-1}$. Since $\alpha_{1n}f^{(1)}+\beta_{1n}f^{(0)}=0$, this implies that $\alpha_{1n}f^{(1)}=x_1$ or $x_1+\dots +x_{n-1}$, which is impossible.
\end{proof}
\begin{corollary}
\label{c3.3.1}
We have $Q_2=n-1$.
\end{corollary}
\begin{prop}
\label{p3.4}
The dimension of $\mathcal{L}_{1,c}[3]$ is $\binom{n+1}{3}$.
\end{prop}
\begin{proof}
Expand and look at the coefficient of $z^3$ in $h_{\mathcal{L}_{1,c}}(z)$ via the form from \cite{balagovic2013representations}.
\end{proof}
\begin{remark}
This implies that under the substitution $x_n=x_1+x_2+\dots +x_{n-1}$, there are no singular polynomials in the space of homogeneous degree $3$ polynomials in $\{x_1,x_2,\dots ,x_{n-1}\}$.
\end{remark}
\begin{prop}
\label{p3.5}
The polynomials $R_{ij}=\frac{c+1}{c}\left(x_i^4+x_i^2x_j^2+x_j^4\right)+\left(x_i^3+x_i^2x_j+x_ix_j^2+x_j^3\right)\left(\sum_{k\ne i,j,n}x_k\right)+\left(x_i^2+x_j^2\right)\left(\sum_{l\ne i,j,n}x_l^2\right)+(x_i+x_j)\left(\sum_{a,b\ne i,j,n; a\ne b}x_a^2x_b\right)$ are singular for all $i,j\in[n-1]$.
\end{prop}
\begin{remark}
Note that $R_{ij}$ can be rewritten as $$R_{ij}=\frac{1}{c}\left(x_i^4+x_i^2x_j^2+x_j^4\right)+x_i^2x_j^2+(x_i+x_j)\sum_{k\ne i,j}x_k^3.$$
\end{remark}
\begin{proof}
We will prove that $f=cR_{12}$ is singular (the rest are the same by symmetry). Let $f^{(0)}=x_1^4+x_1^2x_2^2+x_2^4$ and $f^{(1)}=x_1^2x_2^2+(x_1+x_2)\sum_{k>2}x_k^3=R_{ij}-\frac{1}{c}f^{(0)}$.

We compute the action of $D_{y_j}=\partial_{x_j}-cD_j$ where $D_j=\sum_{k\ne j}\frac{1-\sigma_{1j}}{x_j-x_k}$. We have that $$D_{y_j}cR_{12}=\left(\partial_{x_j}f^{(0)}\right)+c\left(D_j f^{(0)}+\partial_{x_j}f^{(1)}\right)+c^2\left(D_j f^{(1)}\right),$$ so we are interested in the action upon each degree (when it is viewed as a polynomial in $c$ with coefficients in the ring $\Bbbk[x_1,x_2,\dots ,x_n]$). Computing the action of $D_{y_1}$, we have
\begin{align*}
\partial_{x_1}f^{(0)} &=0,\\
\sum_{k\ne 1}\frac{1-\sigma_{1k}}{x_1-x_k}f^{(0)} &=x_1^2x_2+x_1x_2^2+\sum_{k=1}^{n}x_k^3,\\
\partial_{x_1}f^{(1)} &=\sum_{k>2}x_k^3,\\
\sum_{k\ne 1}\frac{1-\sigma_{1k}}{x_1-x_k}f^{(1)} &=\sum_{k>2}\left[x_2^2(x_1+x_k)+x_1^3+x_1^2x_k+x_1x_k^2+x_k^3+x_2^3+x_2\left(x_1^2+x_1x_k+x_k^2\right)+\sum_{j\in [n]}x_j^3\right].\\
\end{align*}
From this, we note that (still viewing $D_{y_1}cR_{12}$ as a polynomial in $c$ with coefficients in $\Bbbk[x_1,\dots ,x_n]$) the constant term is $0$. The coefficient of $c$ is then $x_1^3+x_1^2x_2+x_1x_2^2+x_2^3$. As for the coefficient of $c^2$, we sum and using the fact that $x_n=x_1+\dots +x_{n-1}$, we ultimately obtain $0$. Hence $D_{y_1}R_{12}=x_1^3+x_1^2x_2+x_1x_2^2+x_2^3$.

Note that the action of $D_{y_2}$ is identical to the action of $D_{y_1}$ by symmetry, so it suffices to compute the action of $D_{y_3}$ (since all other $j>2$ are analogous). Then
\begin{align*}
\partial_{x_3}f^{(0)} &=0,\\
\sum_{j\ne 3}\frac{1-\sigma_{3j}}{x_3-x_j}f^{(0)} &=x_1^3+x_1^2x_2+x_1x_2^2+x_2^3+(x_1+x_2)x_3^2,\\
\partial_{x_3}f^{(1)} &=(x_1+x_2)x_3^2,\\
\sum_{j\ne 3}\frac{1-\sigma_{3j}}{x_3-x_j}f^{(1)} &=\sum_{j\ne 3}\frac{1-\sigma_{3j}}{x_3-x_j}\left(x_1^4+x_1^3x_2+x_1^2x_2^2+x_1x_2^3+x_2^4+(x_1+x_2)\sum_{j\in [n]}x_j^3\right),\\
&=\frac{1-\sigma_{13}}{x_3-x_1}\left(x_1^4+x_1^3x_2+x_1^2x_2^2+x_1x_2^3+x_2^4\right)+\frac{1-\sigma_{23}}{x_3-x_2}\left(x_1^4+x_1^3x_2+x_1^2x_2^2+x_1x_2^3+x_2^4\right)=0.\\
\end{align*}
Via the discussion above, we again see that the constant term in $D_{y_3}cR_{12}$ is $0$. The coefficient of $c$ is $x_1^3+x_1^2x_2+x_1x_2^2+x_2^3$, and the coefficient of $c^2$ is $0$. Thus $D_{y_k}=x_1^3+x_1^2x_2+x_1x_2^2+x_2^3$ for all $k\in [n]$, and hence all Dunkl operators $D_{y_i-y_j}$ send all $R_{kl}$ to $0$.
\end{proof}

\begin{prop}
\label{p3.6}
The singular polynomials $R_{ij}$ in degree $4$ are linearly independent for $i,j<n$.
\end{prop}
\begin{proof}
Consider some linear combination of the $R_{ij}$ and consider the lowest gradation by powers of $c$. This is comprised of the $x_i^4+x_i^2x_j^2+x_j^4$ terms. But the $x_i^2x_j^2$ part is unique to each $R_{ij}$, hence they are linearly independent.
\end{proof}
\begin{prop}
\label{p3.7}
The dimension of $\mathcal{L}_{1,c}[4]$ is $\binom{n+2}{4}-\binom{n-1}{2}$.
\end{prop}
\begin{proof}
This is equivalent to the claim that no more singular polynomials exist in $\mathcal{M}_{1,c}[4]$. After the substitution for $x_n$, consider such a polynomial $f$. Then $\alpha_{ij}f^{(0)}=0$, so no terms of the form $x_i^3x_j$ exist in $f^{(0)}$. Now subtract copies of $cR_{ij}$ to remove the terms of the form $x_i^2x_j^2$, so that without loss of generality $f^{(0)}=x_1^4+x_2^4+\dots +x_C^4$. Then $\alpha_{1n}f^{(1)}+\beta_{1n}f^{(0)}=0$, or $$0=\partial_{x_1}f^{(1)}+(1+C)\left(x_2^3+\dots +x_{n-1}^3+\sum_{i,j<n}x_i^2x_j\right)+\left(x_1^3+\dots +x_{n-1}^3\right)$$$$+x_1^2(x_1+\dots +x_C)+x_1(x_1+\dots +x_C)^2+(x_1+\dots +x_{n-1})^2(x_1+\dots +x_C)+(x_1+\dots +x_{n-1})(x_1+\dots +x_C)^2.$$ However, regardless of the parity of $C$, there must remain an $x_1^3$ term, which is impossible to produce in $\partial_{x_1}f^{(1)}$. Since by assumption $f^{(0)}\ne 0$, we have the result.
\end{proof}
\begin{corollary}
\label{c3.7.1}
We have $Q_4=n-1$.
\end{corollary}
\begin{prop}
\label{p3.8}
The dimension of $\mathcal{L}_{1,c}[5]$ is $\binom{n+3}{5}-(n-1)\binom{n-1}{2}$.
\end{prop}
\begin{proof}
Expand and look at the coefficient of $z^5$ in $h_{\mathcal{L}_{1,c}}(z)$ using the form from \cite{balagovic2013representations}. It turns out to be $\binom{n-1}{5}+(n-1)\binom{n-1}{3}+(n-1)^2$, which is equivalent to that expression.
\end{proof}
\begin{remark}
This means that the only polynomials in degree $\ker \mathcal{B}[5]$ are linear combinations of $x_{\ell}R_{ij}$ for $i,j,\ell<n$ and $i\ne j$.
\end{remark}

The key theorem we will now introduce allows us to answer the question of whether a specific polynomial is in $\ker \mathcal{B}$ for all odd $n$.
\begin{theorem}
\label{th3.9}
Let $f$ be a (homogeneous) polynomial in $k$ variables (for simplicity, say $x_1,x_2,\dots ,x_k$). Define $G=\deg f$ and $S$ to be the maximal exponent of any of the variables. Then $f\in\ker \mathcal{B}$ in any $\mathcal{M}_{1,c}(S_n,\mathfrak{h})$ iff $f\in\ker \mathcal{B}$ in any $\mathcal{M}_{1,c}(S_n,\mathfrak{h})$ for all $n \le S+k+G-2$.
\end{theorem}
\begin{proof}
We work in the ring of polynomials in infinite number of variables $\Bbbk[x_1,x_2,\dots ]$ and consider the subring $\Bbbk[x_1,x_2,\dots ,x_n]$. Denote by $D_{y_i}^{(n)}$ the Dunkl operator associated with $\mathcal{H}_{1,c}(\mathfrak{h},S_n)$; i.e., $D_{y_i}^{(n)}=\partial_{x_i}-c\sum_{j\ne i,\hspace{2mm}j\le n}\frac{1-\sigma_{ij}}{x_i-x_j}$. We also denote by $e_s^{(n)}=\sum_{j=1}^{n}x_j^s$ and $e_0^{(n)} = 1$.

It is easy to see that if $i>k$ then $D_{y_i}^{(n)}f=\sigma_{i,k+1}D_{y_{k+1}}^{(k+1)}f$ hence for $i>k$, the value of $D_{y_i}f$ does not depend on $n$. But the result starts to depend on a new variable, namely $x_i$.

Now suppose $i\le k$ Computing, $$D_{y_i}^{(n)}f=\partial_{x_i}f-c\sum_{j\ne i,\hspace{2mm}j\le k}\frac{1-\sigma_{ij}}{x_i-x_j}f-c\sum_{k<j\le n}\frac{1-\sigma_{ij}}{x_i-x_j}f=F(x_1,\dots ,x_k)-c\sum_{k<j\le n}\frac{1-\sigma_{ij}}{x_i-x_j}f,$$ where $F(x_1,\dots ,x_k)$ is a polynomial which depends only on $x_1,x_2,\dots ,x_k$. Now write $f=\sum_{\ell}f_{\ell}x_i^{\ell}$, grading $f$ by degree in $x_i$. Crucially, each $f_{\ell}\in \Bbbk[x_1,x_2,\dots ,x_k]$. Then $$D_{y_i}^{(n)}f=F-c\sum_{k<j\le n}\sum_{\ell}f_{\ell}\left(x_i^{\ell -1}+x_i^{\ell -2}x_j+\dots +x_j^{\ell -1}\right)$$$$=F-c\sum_{\ell}f_{\ell}\left(x_i^{\ell-1}(n-k)+x_i^{\ell-2}\left(e_1^{(n)}-e_1^{(k)}\right)+x_i^{\ell-3}\left(e_2^{(n)}-e_2^{(k)}\right)+\dotsm +\left(e_{\ell -1}^{(n)}-e_{\ell -1}^{(k)}\right)\right).$$ Since $n-k=1-k$ in $\Bbbk$, that term does not depend on $n$, and thus we can write $$D_{y_i}^{(n)}f=\sum_{s=0}^{\deg_{x_i}(f)-1}F_s(x_1,\dots ,x_k)e_s^{(n)},$$ where $\deg_{x_i}(f)$ is the degree of $f$ as a polynomial in $x_i$ and each $F_s(x_1,\dots ,x_k)\in\Bbbk[x_1,x_2,\dots ,x_k]$ (recall that $c\in\Bbbk$).

Now let us examine the action of Dunkl operators $D_{y_j}$ on each term $F_s(x_1,\dots ,x_k)e_s^{(n)}$. Note that $$D_{y_j}^{(n)}\left(F_se_s^{(n)}\right)=D_{y_j}^{(n)}(F)e_s^{(n)}+F\partial_{x_i}\left(e_s^{(n)}\right).$$ If $j\le k$, then $D_{y_j}^{(n)}\left(F_se_s^{(n)}\right)=\sum_{s,t} F_{s,t}(x_1,\dots ,x_k)e_{s}^{(n)}e_{t}^{(n)}$. If $j>k$, then it depends on $x_j$ in an asymmetric way, and thus $D_{y_j}^{(n)}\left(F_se_s^{(n)}\right)=\sum_{s,t} F_{s,t}(x_1,x_2,\dots ,x_k,x_j)e_{s}^{(n)}e_{t}^{(n)}$.
So we can prove the following lemma by induction:
\begin{lemma}
\label{l3.10}
If $f\in k[x_1,\dots,x_k]$, then up to an action of $w \in S_{n-k}$(i.e. up to permuting the rest of the variables) $D_{y_{j_r}}\dots D_{y_{j_1}}f$ can be expressed as:
$$
w \circ D_{y_{j_r}}\dotsm D_{y_{j_1}}f = \sum_{s_j < S} F_{s_1,\dots,s_r}(x_1,\dots,x_{k+r}) e_{s_1}^{(n)}\dotsm e_{s_r}^{(n)} \ ,
$$
where $S= \max(\deg_{x_i}(f))$ and $\max(\deg_{x_i}F_{s_1,\dots,s_r}) \le \max(\deg_{x_i}(f))$.
\end{lemma}
\begin{proof}
Note that the polynomial in the lemma is not homogeneous in $e_s$, since $e_0 = 1$. 

We will prove by induction on $r$. For $r=1$ this follows from the previous discusssion. Suppose we know this for $r-1$. Consider $D_{y_{j_r}}\dotsm D_{y_{j_1}}f$. We know that $$w\circ D_{y_{j_{r-1}}}\dotsm D_{y_{j_1}}f= \sum_{s_j < S} F_{s_1,\dots,s_{r-1}}(x_1,\dots,x_{k+r-1}) e_{s_1}^{(n)}\dotsm e_{s_{r-1}}^{(n)}.$$ We can write:
\begin{align*}
w\circ D_{y_{j_r}} D_{y_{j_{r-1}}}\dotsm D_{y_{j_1}}f=&
D_{y_{w(j_r)}} \circ w \circ D_{y_{j_{r-1}}}\dotsm D_{y_{j_1}}f,\\
=&\sum_{s_j < S} D_{y_{w(j_r)}}[F_{s_1,\dots,s_{r-1}}(x_1,\dots,x_{k+r-1})] e_{s_1}^{(n)}\dotsm e_{s_{r-1}}^{(n)}\\
&+ \sum_{s_j < S} [F_{s_1,\dots,s_{r-1}}(x_1,\dots,x_{k+r-1})] \partial_{w(j_r)}[e_{s_1}^{(n)}\dotsm e_{s_{r-1}}^{(n)}] .\\
\end{align*}

We have two cases. First $w(j_r) > k+r-1$. Then by the discussion before the lemma, the first sum consists of polynomials in variables $x_1,\dots, x_{k+r-1},x_{w(j_r)}$ and the number of symmetric polynomials does not grow, the second sum consists of polynomials in the same number of variables, but the number of symmetric polynomials drops by one. So after acting by $\sigma_{k+r,w(j_r)}$ we obtain the formula we need. Since no new $e_s$ arise it follows that the bound by $S$ still holds in this case. Also since action of Dunkl operators does not increase the maximal degree in the single variable the second assertion also works.

The second case is $w(j_r) \le k+r-1$. In this case the first part of the sum does not depend on any new variables, but we get one new symmetric polynomial in the product. Its index is bounded by maximal degree of $F_{s_1,\dots,s_{r-1}}$ in single variable minus 1, so bounded by $S$. The second sum consists of polynomials depending on the same set of variables, but with one symmetric polynomial erased. So we again obtain the polynomial of the same form. Hence the Lemma holds.
\end{proof}

Now the statement that $f\in\ker \mathcal{B}$ will follow from the fact that by acting by any number of Dunkl operators $D_{y_a-y_b}^{(n)}$ on $f$, we obtain $0$. In particular, when $\Bbbk$ has characteristic $2$, $\ker \mathcal{B}[3]$ consists of only the $0$ polynomial after the substitution $x_n=x_1+\dots +x_{n-1}$. We know that the Dunkl operators have a basis $D_{y_1-y_u}^{(n)}$ for $u=2,3,\dots ,n$. Thus it suffices to check that all sequences $u_1,u_2,\dots ,u_{G-3}\in\{2,3,\dots ,n\}$ satisfy $D_{y_1-y_{u_{G-3}}}^{(n)}D_{y_1-y_{u_{G-4}}}^{(n)}\dotsm D_{y_1-y_{u_1}}^{(n)}f=0$ (after performing the substitution $x_n=x_1+\dots +x_{n-1}$). Using the lemma it follows that
$$w \circ D_{y_1-y_{u_{G-3}}}^{(n)}D_{y_1-y_{u_{G-4}}}^{(n)}\dotsm D_{y_1-y_{u_1}}^{(n)}f=\sum_{s_i\le S} F_{s_1,s_2,\dots ,s_{G-3}}(x_1,\dots ,x_k,x_{k+1},\dots ,x_{k+G-3})e_{s_1}^{(n)}\dotsm e_{s_{G-3}}^{(n)} .$$ 
Note that since we work over characteristic $2$ and we factored out $e_1$ it follows that we have only $e_{s_i}$ with $s_i$ - odd and not $1$.
We can rewrite this as 
$$\sum_{s_i\le S} \Tilde{F}_{s_1,\dots ,s_{G-3}}(x_1,\dots ,x_{k+G-3})\Tilde{e}_{s_1}^{(n)}\dotsm \Tilde{e}_{s_{G-3}}^{(n)},$$
where $\Tilde{e}_s^{(n)}=\sum_{t=k+G-2}^{n}x_t^n$ (remember that each $s_t<S$). Note that the $\Tilde{e}_{s_t}^{(n)}$ are algebraically independent  when $n-k-G+3\ge S-1$, and thus for $n\ge k+G+S-2$, if the value is $0$, then by algebraic independence all the $\Tilde{F}$ are zero, and thus the value is $0$ for all $n$ satisfying $n\ge k+G+S-3$. Hence it suffices to check all combinations of Dunkl operators for all $n$ satisfying $n< k+G+S-3$ and at least one of the values of $n$ for $n\ge k+g+S-3$; if it is zero on all of those cases, then $f\in\ker \mathcal{B}$ for all $n$.
\end{proof}
This theorem thus easily shows the following polynomials are in $\ker \mathcal{B}$.
\begin{prop}
\label{p3.11}
\begin{itemize}
    \item The polynomials $x_i^6\in\ker \mathcal{B}$.
    \item The polynomials $x_i^5x_j^2x_k^2\in\ker \mathcal{B}$.
    \item The polynomials $x_i^4x_j^4\in\ker \mathcal{B}$.
    \item The polynomials $x_i^3x_j^3x_k^3\in\ker \mathcal{B}$.
    \item The polynomials $x_i^2x_j^2x_k^2x_{\ell}^2\in\ker \mathcal{B}$.
\end{itemize}
\end{prop}
\begin{proof}
The proof is exhausting all cases using a computer, as outlined in Theorem~\ref{th3.9}.
\end{proof}
However, its use is not limited to showing that a polynomial is in $\ker \mathcal{B}$. The method of proof of the theorem can also show that a polynomial is not in $\ker \mathcal{B}$.
\begin{prop}
\label{p3.12}
The polynomial $x_1^5x_2\not\in\ker \mathcal{B}$.
\end{prop}
\begin{proof}
We show that $D_{y_1-y_2}D_{y_1-y_2}D_{y_1-y_2}x_1^5x_2=c\left(x_1x_2^2+x_2^3\right)$. Adapting the proof of Theorem~\ref{th3.9}, for all $n$, $D_{y_1-y_2}D_{y_1-y_2}D_{y_1-y_2}x_1^5x_2$ will be a polynomial in $\Tilde{e}_0^{(n)},\Tilde{e}_1^{(n)},\Tilde{e}_3^{(n)}$, where $\Tilde{e}_s^{(n)}=\sum_{j=3}^{n}x_j^s$, with coefficients from $\Bbbk[x_1,x_2]$. By algebraic independence and by the fact that the Dunkl operators do not depend on $n$, if the result is the same up to $n=5+2+6-2=11$, then the coefficients are always the same for any $n$. Checking the action of $\left(D_{y_1-y_2}\right)^3\left(x_1^5x_2\right)$ for $n=3,5,7,9,11$ (with a computer, for example) shows that it always holds.
\end{proof}
\begin{corollary}
\label{c3.12.1}
We have $Q_6\ge 1$.
\end{corollary}
\begin{proof}
Check the coefficient of $z^6$ in the Hilbert series for $\mathcal{L}_{1,c}$ using the form from \cite{balagovic2013representations}, noting that $\dim \mathcal{L}_{0,c}[6]\ge 1$.
\end{proof}
\begin{corollary}
\label{c3.12.2}
We have $\dim \mathcal{L}_{1,c}[n+5]\ge 1$.
\end{corollary}
\begin{proof}
Expand the Hilbert series for $\mathcal{L}_{1,c}$ and check the coefficient of $z^6$.
\end{proof}
\begin{prop}
\label{p3.13}
We have the equality $\dim \mathcal{L}_{1,c}[n+5]=1$.
\end{prop}
\begin{proof}
From Theorem~\ref{th3.9}, it is easy to check that $x_1^3x_2^3x_3^2+c\left(x_2^3x_3^5+x_1x_2^2x_3^5\right)\in\ker \mathcal{B}$. Multiplying this polynomial by $x_1$ yields $x_1^4x_2^3x_3^2+c\left(x_1x_2^3x_3^5+x_1^2x_2^2x_3^5\right)\in\ker \mathcal{B}$. Noting that $x_1^2x_2^2x_3^5\in\ker \mathcal{B}$, we obtain that $x_1^4x_2^3x_3^2+cx_1x_2^3x_3^5\in\ker \mathcal{B}$. If either monomial were in $\ker\mathcal{B}$, then combined with Proposition~\ref{p3.11}, every monomial of degree $n+5$ would be contained in $\ker\mathcal{B}$, hence $\dim \mathcal{L}_{1,c}[n+5]\le 1$. But since $\dim \mathcal{L}_{1,c}[n+5]\ge 1$ by Corollary~\ref{c3.12.2}, equality is achieved.
\end{proof}
\begin{corollary}
\label{c3.13.1}
We have that $Q_6=1$.
\end{corollary}
We will denote by $(s_1,s_2,\dots )$, a monomial whose (nonzero) exponents (of its distinct variables) are $s_1,s_2,\dots $ for $s_1\ge s_2\ge \dots $ and all variables are $x_i$ for $i<n$ (we may simply substitute $x_n=x_1+\dots +x_{n-1}$).
\begin{prop}
\label{p3.14}
For $n\ge 5$, $\dim \mathcal{L}_{1,c}[n+7]=0$.
\end{prop}
\begin{proof}
We proceed by Pigeonhole principle on the exponents of any monomial and use Proposition~\ref{p3.11} to show that such a monomial is contained in $\ker\mathcal{B}$. If the highest degree in a single variable is at least $6$, then by $x_i^6\in\ker \mathcal{B}$, it is in $\ker \mathcal{B}$. If its highest degree in a single variable is $5$, then by the Pigeonhole principle it is either $(5,5,\dots )$ or $(5,4,\dots )$ or $(5,3,3,\dots )$ or $(5,3,2,\dots )$ or $(5,2,2,2,\dots )$, all of which can be formed via Proposition~\ref{p3.11}. If its highest degree in a single variable is $4$, then it is either $(4,4,\dots )$ or $(4,3,3,\dots )$ or $(4,3,2,2,\dots )$ or $(4,2,2,2,\dots )$. If its highest degree in a single variable is $3$, then it must have at least three other variables with exponent at least $2$. If its highest degree in a single variable is $2$, then there must be at least $4$ distinct variables with exponent $2$. Hence every possible monomial is contained in $\ker \mathcal{B}[n+7]$.
\end{proof}
\begin{corollary}
\label{c3.14.1}
We have that $Q_8=0$.
\end{corollary}

Putting these together, we find the Hilbert series for $\mathcal{L}_{1,c}$ for $n\ge 5$. For $n=1$, there is not much to say (and the formula does not apply), and for $n=3$, a quick Sage computation shows that the Hilbert series matches the same form.
\begin{theorem}
\label{th3.15}
The Hilbert series for $\mathcal{L}_{1,c}$ over a field with characteristic $2$ is $$h_{\mathcal{L}_{1,c}}(z)=\left(1+z^2\right)(1+z)^{n-1}\left(1+(n-2)z^2+z^4\right),$$ or alternatively, $$h_{\mathcal{L}_{1,c}}(z)=(1+z)^{n-1}\left(1+(n-1)z^2+(n-1)z^4+z^6\right).$$
\end{theorem}

\subsection*{Acknowledgements}
This project was done under the MIT PRIMES-USA program, which authors would like to thank for this opportunity. We would also like to thank Professor Pavel Etingof for suggesting this project. Finally, we would like to thank Professor Pavel Etingof and Professor Alexander Kirillov Jr. for many helpful discussions on the subject.

\bibliographystyle{alpha}
\bibliography{bibl}
\end{document}